\documentclass[12pt]{amsart}

\usepackage{latexsym, amssymb, amsmath,ulem,soul,esint}

\usepackage{amsfonts, graphicx}
\usepackage{graphicx,color}
\newcommand{\be}{\begin{equation}}
\newcommand{\ee}{\end{equation}}
\newcommand{\beq}{\begin{eqnarray}}
\newcommand{\eeq}{\end{eqnarray}}

\newcommand{\vp}{\varphi}

\newcommand{\ov}[1]{\overline{#1}}

\newtheorem{thm}{Theorem}[section]
\newtheorem{lma}[thm]{Lemma}

\newtheorem{defn}[thm]{Definition}

\theoremstyle{remark}
\newtheorem{rem}{Remark}[section]
\numberwithin{equation}{section}

\def\be{\begin{equation}}
\def\ee{\end{equation}}
\def\bee{\begin{equation*}}
\def\eee{\end{equation*}}

\def\K{K\"ahler }

\def\Rm{\text{\rm Rm}}

\newcommand{\de}{\partial}

\def\ti{\tilde}

\def\re{\text{\rm Re}}
\def\im{\text{\rm Im}}
\def\arccot{\mathrm{arccot}}
\def\PSH{\mathrm{PSH}}

\def\ve{\varepsilon}
\def\a{{\alpha}}

\def\R{\mathbb{R}}
\def\C{\mathbb{C}}

\def\ddb{\sqrt{-1}\partial\bar\partial}

\def\Re{\mathrm{Re}}
\def\Im{\mathrm{Im}}
\def\J{\mathcal{J}}
\def\H{\mathcal{H}}
\def\X{\mathcal{X}}
\def\length{\mathrm{length}}
\def\dbar{\bar{\partial}}
\def\tr{\mathrm{tr}}

\def\psh{\mathrm{psh}}

\begin{document}

\title{Hypercritical deformed Hermitian-Yang-Mills equation}

\author[J. Chu]{Jianchun Chu}
\address[Jianchun Chu]{Department of Mathematics\\ Northwestern University\\ 2033 Sheridan Road\\ Evanston, IL 60208}
\email{jianchun@math.northwestern.edu}

%

\author[M.-C. Lee]{Man-Chun Lee$^1$}
\address[Man-Chun Lee]{Mathematics Institute, Zeeman Building,
University of Warwick, Coventry CV4 7AL; Department of Mathematics, Northwestern University, 2033 Sheridan Road, Evanston, IL 60208}
\email{Man.C.Lee@warwick.ac.uk, mclee@math.northwestern.edu}
\thanks{$^1$Research partially supported by EPSRC grant number P/T019824/1.}

\begin{abstract}
In this work, we study the deformed Hermitian-Yang-Mills equation on compact \K manifold. We introduce the notions of coerciveness and properness of the $\mathcal{J}$-functional on the space of almost calibrated $(1,1)$-forms and show that they are both equivalent to the existence of solutions to the hypercritical deformed Hermitian-Yang-Mills equation.
\end{abstract}

\maketitle

\section{introduction}
Informally, mirror symmetry describes the relation between complex geometry and symplectic geometry on Calabi-Yau manifolds. It predicts a duality between the underlying complex and symplectic structure on a manifold. On the holomorphic side, the deformed Hermitian-Yang-Mills (dHYM) equation is corresponding to the special Lagrangian equation in the setting of the Strominger-Yau-Zaslow mirror symmetry \cite{SYZ96}. This was first appeared in \cite{LeungYauZaslow2001} from the mathematical side drawing from the physics literature \cite{MMMS00}. Analytically, on an $n$-dimensional compact \K manifold $(X^n,\omega)$ associated with a closed real $(1,1)$-form $\a$, the dHYM equation is given by
\begin{equation}\label{intro-dHYM}
\left\{
\begin{array}{ll}
\im\left( e^{-\sqrt{-1}\theta_0}(\a_\varphi+\sqrt{-1}\omega)^n\right)=0,\\[2mm]
\re\left( e^{-\sqrt{-1}\theta_0}(\a_\varphi+\sqrt{-1}\omega)^n\right)>0,
\end{array}
\right.
\end{equation}
where $\a_\varphi=\a+\ddb\varphi$ and $\theta_0$ is a constant. By integrating \eqref{intro-dHYM}, it is easy to see that $\theta_0$ is uniquely determined (modulo $2\pi$) by a cohomological condition:
$$\int_X (\a+\sqrt{-1}\omega)^n\in e^{\sqrt{-1}\theta_0}\cdot \R_{>0}$$
provided that the integral on the left hand side does not vanish. It is believed that the dHYM equation plays a fundamental role in mirror symmetry and its solvability is expected to be related to deep notions of stability in algebraic geometry. We refer the reader to \cite{CollinsShi2020,CollinsXieYau2017} and the references therein for an introduction to the physical and mathematical aspects of the dHYM equation.

On the analytic side, the program of solving the dHYM equation was initiated by Jacob-Yau \cite{JacobYau2017} where they used a parabolic flow to prove the existence of solutions when $(X,\omega)$ has positive bisectional curvature, and the initial data is sufficiently positive. Since then, the solvability of the dHYM equation has been studied extensively. In \cite{CollinsJacobYau2020}, Collins-Jacob-Yau characterized the existence
of solutions to dHYM in the supercritical phase case in the analytic
viewpoint, see also the work of Sz\'ekelyhidi \cite{Szekelyhidi2018} for the study of a very general class of Hessian type equations on compact Hermitian manifolds under similar analytic assumption. It is also conjectured by Collins-Jacob-Yau  \cite{CollinsJacobYau2020} that the solvability to dHYM equation is equivalent to a Nakai-Moishezon type criterion in the supercritical phase case, see \cite{Chen2021,DatarPingali2020,JacobSheu2020,ChuLeeTakahashi2021} for the recent progress to this conjecture. For more related works, we refer readers to \cite{HanYamamoto2019,Takahashi2021,Pingali2019,SchlitzerStoppa2019,HanJin2020,Takahashi2020} and the references therein.

On the other hand, motivated by work of Solomon \cite{Solomon2013}, Thomas \cite{Thomas2001} and Thomas-Yau \cite{ThomasYau2002} in symplectic geometry, Collins-Yau \cite{CollinsYau2021} proposed to study the dHYM equation using an infinite dimensional GIT (Geometric Invariant Theory) approach where they introduced the space \footnote{The formulation taken here is equivalent that in \cite{CollinsYau2021}, see Lemma~\ref{Re Im relationships}.} $\H$ of almost calibrated $(1,1)$-forms in the class $[\a]$:
\[
\H=\left\{\varphi\in C^\infty(X)~|~ \re\left(e^{-\sqrt{-1}\theta_0}\left(\a_\varphi+\sqrt{-1}\omega \right)^n \right)>0 \right\}.
\]
The space $\H$ is a (possibly empty) open subset of the space of smooth, real valued functions on $X$, and hence inherits the structure of an infinite dimensional manifold. On $\H$, Collins-Yau \cite{CollinsYau2021} also introduced the notion of $\ve$-geodesic and the $\J$-functional, and proved (1) $\J$-functional is convex along the $\ve$-geodesic for $\ve\geq 0$; (2) $\vp$ is the critical point of $\J$ if and only if $\vp$ is the solution of the dHYM equation. Furthermore, it was later shown by the authors and Collins \cite{ChuCollinsLee2020} that $\H$ admits a Riemannian structure and a Levi-Civita connection with non-positive sectional curvature. In fact, this is closely related, by mirror symmetry, to the work \cite{Solomon2013} of Solomon on the space of positive (or almost calibrated) Lagrangians. We refer readers to the arxiv version \cite{CollinsYau2018} of \cite{CollinsYau2021} for more discussion on the relation between the GIT approach and the algebraic obstruction on the existence.

In this work, we will continue the study of the existence problem to the dHYM equation from the viewpoint of $\H$. We assume $\H$ to be non-empty so that $\theta_0$ is well-defined modulo $2\pi$, see \cite{JacobYau2017,CollinsXieYau2017}. Without loss of generality, we will assume $0\in \H$. And we will work on the case when $[\a]$ has hypercritical phase, i.e., $\theta_0\in (0,\frac{\pi}{2})$. The following is an equivalent description of space $\H$ (see \eqref{equivalent descriptions of H}):
\[
\H=\left\{\,\vp\in C^{\infty}(X)~\big|~Q_{\omega}(\alpha_{\vp})
\in\left(\theta_{0}-\frac{\pi}{2},\theta_{0}+\frac{\pi}{2}\right)\,\right\},
\]
where $Q_\omega(\a_\varphi)$ is the phase operator, see \eqref{def-PQ}. Here we used the notion of phase operator to emphasis the branch we are using.

The $\J$-functional should be understood as the Kempf-Ness functional in the GIT framework, and thus plays an important role in the study of the solvability for the dHYM equation. In fact, in K\"ahler geometry, the properness or coercivity of certain functional is usually related to the existence of canonical metric. For example, Tian's properness conjecture \cite{Tian1994,Tian2000} predicts that the existence of constant scalar curvature K\"ahler (cscK) metric on compact K\"ahler manifold $(X,\omega)$ is equivalent to the properness of the Mabuchi $K$-energy defined on the space of K\"ahler potentials:
\[
\H_{\psh} = \{\varphi\in C^\infty(X)~|~\omega+\ddb\varphi>0\}.
\]
There are many progress toward this conjecture in the case of K\"ahler-Einstein metrics. In \cite{DingTian1992}, Ding-Tian proved one direction (properness implies existence). When $X$ has trivial automorphism group, Tian \cite{Tian1997} confirmed the converse direction (existence implies properness) under a technical assumption. This assumption was later removed by Tian-Zhu \cite{TianZhu2000}. When $X$ has non-trivial automorphism group, Darvas-Rubinstein \cite{DarvasRubinstein2017} gave an counterexample to the direction (existence implies properness).
By rephrasing the properness of the Mabuchi $K$-energy in terms of $d_{1}$-distance of $\H_{\psh}$, they proposed a modified version of Tian's conjecture, and confirmed it in the K\"ahler-Einstein metric case.

For the general cscK metric case of the above modified conjecture, one direction (existence implies properness) was established by Berman-Darvas-Lu \cite{BDL2020}, while the converse direction (properness implies existence) was proved by Chen-Cheng \cite{ChenCheng2017,ChenCheng2018-1,ChenCheng2018-2}. In particular, if the automorphism group is discrete, the existence of cscK metric will be implied by the properness the Mabuchi $K$-energy in terms of $d_{1}$-distance of $\H_{\psh}$. They also showed that cscK metric exists if and only if the Mabuchi $K$-energy is coercive. For the closely related $J$-equation, the similar result was obtained by Collins-Sz\'ekelyhidi \cite{CollinsSzekelyhidi2017}.

Inspired by the above mentioned works, in order to understand the existence problem of the dHYM equation deeply, we introduce the notions of coercivity and properness of $\mathcal{J}$-functional.
\begin{defn}
We say the $\J$-functional is coercive if there exists constants $\delta,C>0$ such that for any $\vp\in\H$,
\[
\J(\vp) \geq \delta\int_{X}(-\vp)
\left(\Im(\alpha_{\vp}+\sqrt{-1}\omega)^{n}-\Im(\alpha+\sqrt{-1}\omega)^{n}\right)-C.
\]
\end{defn}
This is analogous to \cite[Definition 20]{CollinsSzekelyhidi2017} in the study of $J$-equation since $\im(\alpha_{\vp}+\sqrt{-1}\omega)^{n}$ on $X$ is a natural volume form for $\varphi\in \H$. To define the properness, we will need to introduce the $d_p$ pseudo-distance on $\H$ for $p\geq 1$ which is a generalization of the $d_{2}$-distance considered in \cite{ChuCollinsLee2020}.

\medskip

For $p\geq 1$ and  a smooth path $\varphi(t),t\in [0,1]$ in $\H$, define
\[
E_{p}(\vp,t) = \int_{X}|\vp_{t}|^{p}\,\Re\left(e^{-\sqrt{-1}{\theta}_{0}}(\alpha_{\vp}+\sqrt{-1}\omega)^{n}\right), \;\;
\length_{p}(\varphi) = \int_{0}^{1}E_{p}^{\frac{1}{p}}(\vp,t)dt.
\]
For any $\vp_{0},\vp_{1}\in\H$,  the $d_{p}$-distance between $\varphi_0,\varphi_1$ is defined to be
\[
d_{p}(\vp_{0},\vp_{1})
= \inf\left\{\length_{p}(\vp)\,|\,\text{$\vp$ is a smooth path in $\H$
connecting $\vp_{0}$ and $\vp_{1}$}\right\}.
\]

We will show that $d_p$ is indeed a distance on $\H$ and can be realized by the limit of $\ve$-geodesic as $\ve\rightarrow0$. This might be of independent interest on its own.
\begin{thm}\label{H dp theorem}
For each $p\geq1$, the space $(\H,d_{p})$ is a metric space. For any $\vp_{0},\vp_{1}\in\H$, let $\vp^{\ve}$ be the $\ve$-geodesic connecting $\vp_{0}$ and $\vp_{1}$. Then
\[
d_{p}(\vp_{0},\vp_{1}) = \lim_{\ve\rightarrow0}\length_{p}(\vp^{\ve}).
\]
\end{thm}

Since we are working on the hypercritical phase, we can rewrite $\H$ as
\[
\H= \bigcup_{c\in(0,\theta_{0})}\H_{c},
\]
where
\[
\H_{c} = \left\{\,\vp\in\H~\big|~
\Im\left(e^{-\sqrt{-1}c}(\alpha_{\vp}+\sqrt{-1}\omega)^{n}\right)>0\,\right\}.
\]
Clearly, we have $\H_{c_{2}} \subset \H_{c_{1}}$ for $c_{1}\leq c_{2}$ and $\H=\H_0$.
%
In other words, the volume form $\Im(\alpha_{\vp}+\sqrt{-1}\omega)^{n}$ is uniformly positive on the subspace $\H_{c}\subset\H$ and its positivity increases with $c$, see \eqref{equ-H-c}.

We define the properness of $\J$ as follows.

\begin{defn}\label{definition of properness}~
\begin{enumerate}\setlength{\itemsep}{1mm}
\item[(i)] We say the $\J$-functional is weakly proper if for each $c\in(0,\theta_{0})$ there exist constants $\delta_{c},A_{c}>0$ such that for any $\vp\in\H_{c}$ with $\sup_{X}\vp=0$,
\[
\J(\vp) \geq \delta_{c}d_{1}(\vp,0)-A_{c}.
\]
\item[(ii)] We say the $\J$-functional is proper if there exist constants $\delta,A>0$ such that the following holds. For any $c\in(0,\theta_{0})$ and $\vp\in\H_{c}$ with $\sup_{X}\vp=0$,
\[
\J(\vp) \geq \delta\sin(c)d_{1}(\vp,0)-A.
\]
\end{enumerate}
\end{defn}

Let us explain the motivation of introducing the subspace $\H_{c}$ and Definition \ref{definition of properness}. Unlike the definition of properness in the K\"ahler setting, the constant $\delta_{c}$ and $\delta\sin(c)$ in Definition \ref{definition of properness} depend on $c$, which is not uniform for all $\vp\in\H$. The reason of this difference is that there is only one canonical volume form for each $\vp\in\H_{psh}$, while there are two canonical volume forms for each $\vp\in\H$.

More precisely, in the K\"ahler setting, for any $\vp\in\H_{\psh}$, $\omega_{\vp}=\omega+\ddb\vp$ is a K\"ahler metric and so the most natural volume form is $\omega_{\vp}^{n}$. However, the situation becomes more subtle in the dHYM setting. For any $\vp\in\H$, there two available volume forms which are both natural. The first one is
$\re\big(e^{-\sqrt{-1}\theta_0}(\a_\varphi+\sqrt{-1}\omega)^n \big)$, which appears in the definition of $\H$. This volume form is used to define the Riemannian structure and $d_{p}$-distance of $\H$, and so is the canonical volume form in the viewpoint of geometry. The second one is $\Im(\omega+\sqrt{-1}\alpha_{\vp})^{n}$, which appears in the definitions of functionals $\J$ and $\J_{\ve}$. Recall the $\J$-functional serves as the Kempf-Ness functional in the GIT framework. This volume form should be understood as the canonical one in the viewpoint of analysis. The quotient of these two volume forms can be computed as follows:
\[
\frac{\Im(\alpha_{\vp}+\sqrt{-1}\omega)^{n}}
{\re\big(e^{-\sqrt{-1}\theta_0}(\a_\varphi+\sqrt{-1}\omega)^n \big)}
= \frac{\sin(Q_{\omega}(\alpha_{\vp}))}{\cos(Q_{\omega}(\alpha_{\vp})-\theta_{0})}.
\]
On the other hand, roughly speaking, one equivalent characterization of the properness is that $d_{1}$-distance can be controlled by $\J$-functional. This means the positive lower bound of the above quotient should be as the constant $\delta$ in the definition of properness. Recall that $Q_{\omega}(\alpha_{\vp})\in(0,\theta_{0}+\frac{\pi}{2})$ and $\theta_{0}\in(0,\frac{\pi}{2})$. It is clear that there is no positive uniform lower bound for all $\vp\in\H$, especially when $Q_{\omega}(\alpha_{\vp})$ is close to zero. This non-uniformity is the main difference between the K\"ahler and dHYM settings, which motivates us to consider the subspace $\H_{c}=\{\vp\in C^{\infty}(X)~\big|~Q_{\omega}(\alpha_{\vp})\in(c,\Theta_{0})\}$ (this is an equivalent definition of $\H_{c}$, see \eqref{equ-H-c}) and introduce Definition \ref{definition of properness}.

With the above notions of properness and coercivity in hand, we now state our criterion of the existence of solutions to the dHYM equation in the hypercritical phase case.

\begin{thm}\label{main theorem}
Let $(X,\omega)$ be an $n$-dimensional compact K\"ahler manifold with $n\geq4$. Suppose that $\alpha$ is a real closed $(1,1)$-form on $X$ so that $[\a]$ has hypercritical phase, then the followings are equivalent:
\begin{enumerate}\setlength{\itemsep}{1mm}
\item The $\J$-functional is coercive.
\item The $\J$-functional is proper.
\item The $\J$-functional is weakly proper.
\item There exists $\vp\in C^{\infty}(X)$ solving the dHYM equation \eqref{intro-dHYM}.
\end{enumerate}
\end{thm}

\begin{rem}
It is interesting to see whether the above result holds when $n\leq3$. The main obstruction is that the proper twisted dHYM operator $F_{\ve}$ may be not concave, see Lemma \ref{F properties}. 
\end{rem}

\bigskip

The paper is organized as follows: In Section~\ref{Section-Preliminaries}, we will collect some preliminaries of the operators that will be used throughout this work. In Section~\ref{Section-metricspace}, we will study the $d_p$-distance using the $\ve$-geodesic. In particular, we will prove Theorem~\ref{H dp theorem}. In Section~\ref{Section:dHYM flow}, we will study the twisted dHYM flow which is the negative gradient flow of the twisted $\mathcal{J}_{\ve}$-functional. We will study its long-time existence and convergence. In Section~\ref{Section: Technical lemmas}, we will establish some technical lemmas which are important ingredients for the proof of main theorem. In Section \ref{Section: main result 1} and \ref{Section: main result 2}, we will prove our main result, Theorem~\ref{main theorem}.

\section{Preliminaries}\label{Section-Preliminaries}
In this section, we will introduce some basic operators and properties which we will use in this work.

\subsection{Basic operators and properties}
For $\lambda=(\lambda_{1},\ldots,\lambda_{n})\in\mathbb{R}^n$ and $0<\theta<\Theta<\pi$, we define
\[
P(\lambda)=\max_{1\leq k\leq n}\sum_{i\neq k} \arccot(\lambda_i), \;\; Q(\lambda)=\sum_{i=1}^n \arccot(\lambda_i)
\]
and
\[
\Gamma_{\theta,\Theta}=\{ \lambda\in \mathbb{R}^n ~|~P(\lambda)<\theta, \; Q(\lambda)<\Theta\}.
\]
It is clear that the operators $P$ and $Q$ are independent of permutation. For a positive definite $n\times n$ Hermitian matrix $A$ and $n\times n$ Hermitian matrix $B$, we will define
\begin{equation}\label{def-PQ}
P_A(B)=P(\lambda),\;\ Q_A(B)=Q(\lambda)
\end{equation}
where $\lambda$ is the eigenvalue of $B$ with respect to $A$ and
$$\Gamma_{A,\theta,\Theta}=\{B\in  \mathrm{Herm}(n)~|~ P_A(B)<\theta,\; Q_A(B)<\Theta\}.$$

We have the following useful properties of the above operators.
\begin{lma}
For any positive definite matrix $A$ and $0<\theta<\Theta<\pi$,
\begin{enumerate}\setlength{\itemsep}{1mm}
\item[(i)] The space $\ov{\Gamma}_{A,\theta,\Theta}$ is convex;
\item[(ii)] The function $\cot(Q_{A})$ is concave on $\ov{\Gamma}_{A,\theta,\Theta}$.
\end{enumerate}
\end{lma}

\begin{proof}
(i) is proved in \cite[Lemma 5.6 (10)]{Chen2021}. (ii) is a special case of \cite[Lemma 5.6 (9)]{Chen2021} by choosing $f=0$. The difference between $\cot(Q_{A})$ and $F$ in \cite[Lemma 5.6]{Chen2021} is a constant term $\cot(\theta_{0})$, which does not affect the concavity.
\end{proof}

It is sometimes useful to consider
\begin{equation}\label{hat Q}
\hat Q_A(B)=\frac{n\pi}{2}-Q_A(B)=\sum_{i=1}^n \arctan(\lambda_i)
\end{equation}
and $\hat \theta_0=\frac{n\pi}{2}-\theta_0$ where $\hat\theta_0\in \left(\frac{(n-1)\pi}{2},\frac{n\pi}{2}\right)$ since $\theta_0\in (0,\frac{\pi}{2})$.

We have the following equivalent forms of some expressions.
\begin{lma}\label{Re Im relationships}
For any $\vp\in\H$,
\[
\Re\left(e^{-\sqrt{-1}\hat{\theta}_{0}}(\omega+\sqrt{-1}\alpha_{\vp})^{n}\right)
= \Re\left(e^{-\sqrt{-1}\theta_{0}}(\alpha_{\vp}+\sqrt{-1}\omega)^{n}\right),
\]
\[
\Im\left(e^{-\sqrt{-1}\hat{\theta}_{0}}(\omega+\sqrt{-1}\alpha_{\vp})^{n}\right)
= -\Im\left(e^{-\sqrt{-1}\theta_{0}}(\alpha_{\vp}+\sqrt{-1}\omega)^{n}\right).
\]
\end{lma}

\begin{proof}
Let $\lambda_{i}$ be the eigenvalues of $\alpha_{\vp}$ with respect to $\omega$. Then
\begin{equation}\label{Re Im relationships eqn 1}
\begin{split}
 \frac{\Re\left(e^{-\sqrt{-1}\hat{\theta}_{0}}(\omega+\sqrt{-1}\alpha_{\vp})^{n}\right)}{\omega^{n}}
= &
\cos(\hat{Q}_{\omega}(\alpha_{\vp})-\hat{\theta}_{0})\prod_{i=1}^{n}\sqrt{1+\lambda_{i}^{2}} \\
= &
\cos(Q_{\omega}(\alpha_{\vp})-\theta_{0})\prod_{i=1}^{n}\sqrt{1+\lambda_{i}^{2}}\\
= & \frac{\Re\left(e^{-\sqrt{-1}\theta_{0}}(\alpha_{\vp}+\sqrt{-1}\omega)^{n}\right)}{\omega^{n}}.
\end{split}
\end{equation}
The second equality can be proved by the similar argument.
\end{proof}

If we denote $\Theta_0=\theta_0+\frac{\pi}{2}$, then using Lemma \ref{Re Im relationships} and \eqref{Re Im relationships eqn 1}, we have several equivalent descriptions of $\H$:
\begin{equation}\label{equivalent descriptions of H}
\begin{split}
\H = {} & \left\{\,\vp\in C^{\infty}(X)~\big|~
\Re\left(e^{-\sqrt{-1}\theta_{0}}(\alpha_{\vp}+\sqrt{-1}\omega)^{n}\right)>0\,\right\} \\
= {} & \left\{\,\vp\in C^{\infty}(X)~\big|~
\Re\left(e^{-\sqrt{-1}\hat{\theta}_{0}}(\omega+\sqrt{-1}\alpha_{\vp})^{n}\right)>0\,\right\} \\
= {} & \left\{\,\vp\in C^{\infty}(X)~\big|~\hat{Q}_{\omega}(\alpha_{\vp})
\in\left(\hat{\theta}_{0}-\frac{\pi}{2},\hat{\theta}_{0}+\frac{\pi}{2}\right)\,\right\} \\
= {} & \left\{\,\vp\in C^{\infty}(X)~\big|~Q_{\omega}(\alpha_{\vp})
\in\left(\theta_{0}-\frac{\pi}{2},\theta_{0}+\frac{\pi}{2}\right)\,\right\} \\[1mm]
= {} & \left\{\,\vp\in C^{\infty}(X)~\big|~Q_{\omega}(\alpha_{\vp})
\in\left(0,\Theta_{0}\right)\,\right\}.
\end{split}
\end{equation}
Likewise, for $c\in(0,\Theta_{0})$, we can define the perturbed version of space $\H$:
\begin{equation}\label{equ-H-c}
\begin{split}
\H_{c} = {} & \left\{\,\vp\in\H~\big|~
\Im\left(e^{-\sqrt{-1}c}(\alpha_{\vp}+\sqrt{-1}\omega)^{n}\right)>0\,\right\} \\
= {} & \left\{\,\vp\in C^{\infty}(X)~\big|~
Q_{\omega}(\alpha_{\vp})\in\left(c,\theta_{0}+\frac{\pi}{2}\right)\,\right\} \\[1mm]
= {} & \left\{\,\vp\in C^{\infty}(X)~\big|~Q_{\omega}(\alpha_{\vp})\in(c,\Theta_{0})\,\right\}.
\end{split}
\end{equation}
For any $\vp\in\H_{c}$, let $\lambda_{i}$ be the eigenvalues of $\alpha_{\vp}$ with respect to $\omega$. Then
\begin{equation}\label{H c positivity}
\frac{\Im(\omega+\sqrt{-1}\alpha_{\vp})^{n}}{\omega^{n}}
= \sin(Q_{\omega}(\alpha_{\vp}))\prod_{i=1}^{n}\sqrt{1+\lambda_{i}^{2}}
\geq \sin(c)\prod_{i=1}^{n}\sqrt{1+\lambda_{i}^{2}}.
\end{equation}
Hence, the constant $c$ is measuring the positivity of the volume form $\Im(\omega+\sqrt{-1}\alpha_{\vp})^{n}$.

\bigskip

Finally, we note that any element of $\H$ is quasi-plurisubharmonic.
\begin{lma}\label{quasi-plurisubharmonic}
Define $\chi=\tan(\theta_{0})\omega+\alpha$. Then for any $\vp\in\H$,
\[
\chi_{\vp} = \chi+\ddb\vp > 0,
\]
or equivalently, $\vp\in\PSH(X,\chi)$.
\end{lma}

\begin{proof}
Let $\lambda_{1}\geq\ldots\geq\lambda_{n}$ be the eigenvalues of $\alpha_{\vp}$ with respect to $\omega$. Then
\[
\arccot(\lambda_{n}) \leq Q_{\omega}(\alpha_{\vp}) < \Theta_{0} = \theta_{0}+\frac{\pi}{2},
\]
which implies $\lambda_{n}>-\tan(\theta_{0})$ and so $\alpha_{\vp}>-\tan(\theta_{0})\omega$. Hence,
\[
\chi_{\vp} = \chi+\ddb\vp = \tan(\theta_{0})\omega+\alpha_{\vp} > 0.
\]
\end{proof}

\subsection{Twisted operator}
In this subsection, we will introduce the twisted dHYM operator. For $0<\theta<\Theta<\pi$ and $\lambda\in \Gamma_{\theta,\Theta}$, consider $F_{\ve} : \Gamma_{\theta,\Theta}\rightarrow \mathbb{R}$ which is given by
\[
F_\ve (\lambda)=\frac{\re\prod_{k=1}^n\left(\lambda_k+\sqrt{-1} \right)+\ve}{\im\prod_{k=1}^n\left(\lambda_k+\sqrt{-1} \right)}.
\]

The following basic properties of $F_\ve$ are very important in the upcoming discussion.
\begin{lma}\label{F properties}
For $0<\theta<\Theta<\pi$, there exists $\ve_{0}(\theta,\Theta,n)$ such that if $n\geq4$ and $\ve\in(0,\ve_{0})$, then the map $F_\ve$ satisfies
\begin{enumerate}\setlength{\itemsep}{1mm}
\item[(i)] For each $i$, $\frac{\de F_{\ve}}{\de\lambda_{i}}>0$;
\item[(ii)] If $\lambda_{i}\geq\lambda_{j}$, then
$\frac{\de F_{\ve}}{\de\lambda_{i}}\leq \frac{\de F_{\ve}}{\de\lambda_{j}}$;
\item[(iii)] $F_{\ve}$ is concave.
\end{enumerate}
\end{lma}

\begin{proof}
This lemma is a special case of \cite[Lemma 5.6]{Chen2021} by choosing $f=-\ve$. The difference between $F_{\ve}$ and $F$ of \cite[Lemma 5.6]{Chen2021} is a constant term $\cot(\theta_{0})$, which does not affect the derivative and concavity.
\end{proof}

For any $\ve\in[0,1]$ and $\vp\in\H$, let $\lambda_{i}$ be the eigenvalues of $\alpha_{\vp}$ with respect to $\omega$. We define the operator $F_{\omega,\ve}$ by $F_{\omega,\ve}(\alpha_{\vp})=F_{\ve}(\lambda)$, or equivalently,
\[
F_{\omega,\ve}(\alpha_{\vp}) = \frac{\Re(\alpha_{\vp}+\sqrt{-1}\omega)^{n}+\ve\omega^{n}}{\Im(\alpha_{\vp}+\sqrt{-1}\omega)^{n}}.
\]
It is clear from the definition that $F_{\omega,\ve}$ is independent of permutation of $\lambda$. Denote the first and second derivatives of $F_{\omega,\ve}$ by $F_{\omega,\ve}^{i\ov{j}}$ and $F_{\omega,\ve}^{i\ov{j},p\ov{q}}$. For any $x_{0}\in X$, if we choose a coordinate such that at $x_0$,
\[
\omega_{i\ov{j}} = \delta_{ij}, \ \
(\alpha_{\vp})_{i\ov{j}} = \lambda_{i}\delta_{ij},
\]
then we have (see e.g. \cite{Andrews1994,Gerhardt1996,Spruck2005})
\begin{equation}\label{derivatives of F}
F_{\omega,\ve}^{i\ov{j}} = \frac{\de F_{\ve}}{\de\lambda_{i}}\delta_{ij}, \quad
F_{\omega,\ve}^{i\ov{j},p\ov{q}}
= \frac{\de^{2}F_{\ve}}{\de\lambda_{i}\de\lambda_{p}}\delta_{ij}\delta_{pq}
+\frac{\frac{\de F_{\ve}}{\de\lambda_{i}}-\frac{\de F_{\ve}}{\de\lambda_{j}}}{\lambda_{i}-\lambda_{j}}
(1-\delta_{ij})\delta_{iq}\delta_{jp}.
\end{equation}
If $\lambda_{i}=\lambda_{j}$, we will then regard the quotient appeared on the last term as the limit. For notational convenience, we will denote
$$\frac{\partial F_\ve}{\partial \lambda_i}=(F_\ve)_i \quad\text{and}\quad \frac{\partial^2 F_\ve}{\partial \lambda_i\partial\lambda_j}=(F_\ve)_{ij}.$$

\subsection{Twisted $\mathcal{J}$-functional and gradient flow}
Motivated by the work in \cite{CollinsSzekelyhidi2017}, we will study the twisted $\mathcal{J}$-functional. For $\ve\in[0,1]$ and $\vp\in\H$, we define the $\J_{\ve}$-functional on $\H$ by setting $\J_{\ve}(0)=0$ and
\[
\begin{split}
d\J_{\ve}(\vp)(\psi) = {} & -\int_{X}\psi\big(F_{\omega,\ve}(\alpha_{\vp})
-\cot(\theta_{0})-a_{0}\ve\big)\Im(\alpha_{\vp}+\sqrt{-1}\omega)^{n} \\
= {} & -\int_{X}\psi\left(\Re(\alpha_{\vp}+\sqrt{-1}\omega)^{n}+\ve\omega^{n}
-(\cot(\theta_{0})+a_{0}\ve)\Im(\alpha_{\vp}+\sqrt{-1}\omega)^{n}\right),
\end{split}
\]
where $\vp\in\H$, $\psi\in T_{\vp}\H=C^{\infty}(X)$ and the constant $a_{0}$ is given by
\[
a_{0} = \frac{\int_{X}\omega^{n}}{\int_{X}\Im(\alpha+\sqrt{-1}\omega)^{n}} > 0.
\]
Clearly, for any $C\in \mathbb{R}$, we have $\J_{\ve}(\vp+C) = \J_{\ve}(\vp)$.

We now compare the $\J_{\ve}$-functional and the $\J$-functional introduced by Collins-Yau \cite{CollinsYau2021}. Recall that $\J$ is defined by
\[
\J(0)=0, \quad
d\J(\vp)(\psi) = -\int_{X}\psi
\Im\left(e^{-\sqrt{-1}\hat{\theta}_{0}}(\omega+\sqrt{-1}\alpha_{\vp})^{n}\right).
\]

\begin{lma}\label{J J 0 relationship}
For any $\vp\in\H$, we have
\[
\J(\vp) = \sin(\theta_{0})\J_{0}(\vp).
\]
\end{lma}

\begin{proof}
By Lemma \ref{Re Im relationships}, we compute
\[
\begin{split}
d\J(\vp)(\psi) = {} & \int_{X}\psi
\Im\left(e^{-\sqrt{-1}\theta_{0}}(\alpha_{\vp}+\sqrt{-1}\omega)^{n}\right) \\
= {} & \int_{X}\psi\left[\cos(\theta_{0})\Im(\alpha_{\vp}+\sqrt{-1}\omega)^{n}
-\sin(\theta_{0})\Re(\alpha_{\vp}+\sqrt{-1}\omega)^{n}\right] \\
= {} & -\sin(\theta_{0})\int_{X}\psi\left[\Re(\alpha_{\vp}+\sqrt{-1}\omega)^{n}
-\cot(\theta_{0})\Im(\alpha_{\vp}+\sqrt{-1}\omega)^{n}\right] \\[2.5mm]
= {} & \sin(\theta_{0})d\J_{0}(\vp)(\psi).
\end{split}
\]
Since $\J(0)=\J_{0}(0)=0$, we have $\J(\vp)=\sin(\theta_{0})\J_{0}(\vp)$.
\end{proof}

From the variational viewpoint, it will be natural to consider a parabolic flow which is the negative gradient flow of $\J_\ve$-functional. We therefore consider the following twisted dHYM flow. For any $\ve\in[0,1]$ and $\vp_{0}\in\H$, the twisted dHYM flow is an one parameter family of $\varphi(t)$ such that
\begin{equation}\label{twisted dHYM flow 1}
\left\{
\begin{array}{ll}
\ \partial_t \varphi = F_{\omega,\ve}(\alpha_{\vp})-\cot(\theta_{0})-a_{0}\ve, \\[2mm]
\ \vp(0) = \vp_{0}.
\end{array}
\right.
\end{equation}
Clearly, the twisted dHYM flow \eqref{twisted dHYM flow 1} is the negative gradient flow of the functional $\J_{\ve}$ whenever $\varphi(t)\in \H$:
\[
\frac{d\J_{\ve}(\vp)}{dt}
= -\int_{X}\left(F_{\omega,\ve}(\alpha_{\vp})-\cot(\theta_{0})-a_{0}\ve\right)^{2}
\Im(\alpha_{\vp}+\sqrt{-1}\omega)^{n}<0.
\]

On the other hand, by Lemma \ref{F properties} and \eqref{derivatives of F}, there is $\ve_0(\varphi_0,\a,\omega,X)>0$ such that if $\ve<\ve_0$, then the twisted dHYM flow \eqref{twisted dHYM flow 1} is strictly parabolic at $t=0$ and hence it admits a short-time solution on $X$. In section~\ref{Section:dHYM flow}, we will study the twisted dHYM flow in more details.

\section{Metric space $(\H,d_{p})$}\label{Section-metricspace}
In this section, we will introduce the $d_{p}$-distance on $\H$ and discuss its properties which ultimately prove Theorem \ref{H dp theorem}. This is motivated by the work of Darvas \cite{Darvas2015} which considered the space of \K potentials. We will also derive some $d_{p}$-distance inequalities for later purpose.

\subsection{$\ve$-geodesic}
In this subsection, we will recall the notion and properties of the $\ve$-geodesic introduced by Collins-Yau in \cite{CollinsYau2021}. Consider
\[
\X = X\times A, \quad A = \{z\in\mathbb{C}~|~e^{-1}\leq|z|\leq1\}.
\]
Denote the projection map from $\X$ to $X$ by $\pi$. Let $\de,\dbar$ and $D,\bar{D}$ be the complex differential operators on $X$ and $\X$ respectively. For a smooth path $\vp(t),t\in[0,1]$ on $X$, define the function on $\X$ by
\[
\Phi(x,z) = \vp(x,-\log|z|).
\]
Following \cite{CollinsYau2021}, the path $\vp$ is said to be a geodesic from $\varphi_0\in\H$ to $\varphi_1\in\H$ if and only if the function $\Phi$ satisfies
\[
\begin{cases}
\Im\left(e^{-\sqrt{-1}\hat{\theta}_{0}}\left(\pi^{*}\omega
+\sqrt{-1}(\pi^{*}\alpha+\sqrt{-1}D\bar{D}\Phi^{\ve})\right)^{n+1}\right) = 0, \\[2mm]
\Re\left(e^{-\sqrt{-1}\hat{\theta}_{0}}
\left(\omega+\sqrt{-1}\alpha_{\Phi^{\ve}}\right)^{n}\right) > 0, \\[2mm]
\Phi^{\ve}|_{\{|z|=1\}} = \vp_{0}, \quad \Phi^{\ve}|_{\{|z|=e^{-1}\}} = \vp_{1}.
\end{cases}
\]
Since $\pi^{*}\omega$ is a degenerate metric on $\X$, the above equation is degenerated. To overcome the degeneracy, Collins-Yau \cite{CollinsYau2021} considered the K\"ahler metric $\hat{\omega}_{\ve}=\pi^{*}\omega+\ve^{2}\sqrt{-1}dz\wedge d\bar{z}$ on $\X$  for $\ve>0$, and introduced the $\ve$-geodesic equation:
\begin{equation}\label{ve geodesic eqn 1}
\begin{cases}
\Im\left(e^{-\sqrt{-1}\hat{\theta}_{0}}\left(\hat{\omega}_{\ve}
+\sqrt{-1}(\pi^{*}\alpha+\sqrt{-1}D\bar{D}\Phi^{\ve})\right)^{n+1}\right) = 0, \\[2mm]
\Re\left(e^{-\sqrt{-1}\hat{\theta}_{0}}
\left(\omega+\sqrt{-1}\alpha_{\Phi^{\ve}}\right)^{n}\right) > 0, \\[2mm]
\Phi^{\ve}|_{\{|z|=1\}} = \vp_{0}, \quad \Phi^{\ve}|_{\{|z|=e^{-1}\}} = \vp_{1}.
\end{cases}
\end{equation}
Using the operator $\hat{Q}_{\hat{\omega}_{\ve}}$ in \eqref{hat Q}, we may rewrite the above equation as
\begin{equation}\label{ve geodesic eqn 2}
\begin{cases}
\hat{Q}_{\hat{\omega}_{\ve}}(\pi^{*}\alpha+\sqrt{-1}D\bar{D}\Phi^{\ve})
= \hat{\theta}_{0}, \\[2mm]
\Re\left(e^{-\sqrt{-1}\hat{\theta}_{0}}
\left(\omega+\sqrt{-1}\alpha_{\Phi^{\ve}}\right)^{n}\right) = 0, \\[2mm]
\Phi^{\ve}|_{\{|z|=1\}} = \vp_{0}, \quad \Phi^{\ve}|_{\{|z|=e^{-1}\}} = \vp_{1}.
\end{cases}
\end{equation}
This is a Dirichlet problem for the specified Lagrangian phase equation on $\X$. If \eqref{ve geodesic eqn 2} admits a smooth solution, then the maximum principle shows
\[
\Phi^{\ve}(\cdot,z) = \Phi^{\ve}(\cdot,|z|).
\]
Therefore, it is natural to consider
\[
\vp^{\ve}(\cdot,t) = \Phi^{\ve}(\cdot,e^{-t}), \quad t\in[0,1].
\]
And we say that the path $\vp^{\ve}$ is the $\ve$-geodesic joining from $\vp_{0}$ to $\vp_{1}$. By \cite[Lemma 2.3]{ChuCollinsLee2020}, the $\ve$-geodesic $\vp^{\ve}$ satisfies
\begin{equation}\label{ve geodesic eqn 3}
\begin{split}
& \vp_{tt}^{\ve}\Re\left(e^{-\sqrt{-1}\hat{\theta}_{0}}
\left(\omega+\sqrt{-1}\alpha_{\vp^{\ve}}\right)^{n}\right) \\
& +n\sqrt{-1}\de\vp_{t}^{\ve}\wedge\dbar\vp_{t}^{\ve}
\wedge\Im\left(e^{-\sqrt{-1}\hat{\theta}_{0}}
\left(\omega+\sqrt{-1}\alpha_{\vp^{\ve}}\right)^{n-1}\right) \\
& = -4e^{-2t}\ve^{2}\Im\left(e^{-\sqrt{-1}\hat{\theta}_{0}}
\left(\omega+\sqrt{-1}\alpha_{\vp^{\ve}}\right)^{n}\right).
\end{split}
\end{equation}

In \cite{CollinsYau2021}, Collins-Yau solved the Dirichlet problem \eqref{ve geodesic eqn 1} (or equivalently, \eqref{ve geodesic eqn 2}) and established the weak $C^{1,1}$ estimate of the solution. Based on this work, the authors and Collins \cite{ChuCollinsLee2020} extended it to the full $C^{1,1}$ estimate.

\begin{thm}[Theorem 1.2 of \cite{CollinsYau2021}, Theorem 6.1 of \cite{ChuCollinsLee2020}]\label{ve geodesic existence and estimates}
For any $\vp_{0},\vp_{1}\in\H$, there exists a unique solution of \eqref{ve geodesic eqn 1} (or equivalently, \eqref{ve geodesic eqn 2}). Furthermore, there exists a constant $C(\vp_{0},\vp_{1},\alpha,\omega,X)$ such that
\[
\sup_{M}\left(|\Phi^{\ve}|+|D\Phi^{\ve}|+|D^{2}\Phi^{\ve}|\right) \leq C,
\]
or equivalently,
\begin{equation}\label{ve geodesic estimates}
\sup_{X}\left(|\vp^{\ve}|+|\vp_{t}^{\ve}|+|\vp_{tt}^{\ve}|
+|\nabla\vp^{\ve}|+|\nabla^{2}\vp^{\ve}|\right) \leq C.
\end{equation}
\end{thm}

For later use, we collect more estimates of $\ve$-geodesic.
\begin{lma}[Lemma 3.1 and 3.4 of \cite{ChuCollinsLee2020}]\label{ve geodesic more estimates}
For any $\vp_{0},\vp_{1}\in\H$, let $\vp^{\ve}$ be the $\ve$-geodesic connecting $\vp_{0}$ and $\vp_{1}$. Then there exists $C(\vp_{0},\vp_{1},\alpha,\omega,X)$ such that
\begin{enumerate}\setlength{\itemsep}{1mm}
\item[(i)] $\vp_{tt}^{\ve}\geq-C\ve^{2}$;
\item[(ii)] $|\vp_{t}^{\ve}|\leq\|\vp_{0}-\vp_{1}\|_{L^{\infty}}+C\ve^{2}$.
\end{enumerate}
\end{lma}

\begin{proof}
(i) is proved in \cite[Lemma 3.1]{ChuCollinsLee2020}. (ii) follows from the proof of \cite[Lemma 3.4]{ChuCollinsLee2020}.
\end{proof}

For notational convenience, for any $\vp\in\H$, we will define
\begin{equation}\label{Omega vp}
\Omega_{\vp} = \omega+\sqrt{-1}\alpha_{\vp} = \omega+\sqrt{-1}(\alpha+\ddb\vp).
\end{equation}

%

\subsection{$d_{p}$-distance}
Let $\vp(t),t\in[0,1]$ be a smooth path in $\H$. For $p\in[1,\infty)$, we define the $L^{p}$-energy and $L^{p}$-length of $\vp$ by
\[
E_{p}(\vp,t) = \int_{X}|\vp_{t}|^{p}\,\Re(e^{-\sqrt{-1}\hat{\theta}_{0}}\Omega_{\vp}^{n}), \quad
\length_{p}(\varphi) = \int_{0}^{1}E_{p}^{\frac{1}{p}}(\vp,t)dt,
\]
where $\Omega_{\vp}=\omega+\sqrt{-1}\alpha_{\vp}$. For any $\vp_{0},\vp_{1}\in\H$,  the $d_{p}$-distance between $\varphi_0,\varphi_1$ is then given by
\[
d_{p}(\vp_{0},\vp_{1})
= \inf\left\{\length_{p}(\vp)\,|\,\text{$\vp$ is a smooth path in $\H$
connecting $\vp_{0}$ and $\vp_{1}$}\right\}.
\]

In the rest of this subsection, we will prove Theorem \ref{H dp theorem}. Since $|\cdot|^{p}$ is not differentiable at zero for general $p\geq 1$, for any $\delta>0$, we introduce the following approximated $L^{p}$-energy and $L^{p}$-length:
\[
E_{p,\delta}(\vp,t) = \int_{X}\sqrt{\vp_{t}^{2p}+\delta^{2}}
\,\Re(e^{-\sqrt{-1}\hat{\theta}_{0}}\Omega_{\vp}^{n}), \quad
\length_{p,\delta}(\vp) = \int_{0}^{1}E_{p,\delta}^{\frac{1}{p}}(\vp,t)dt,
\]
then for any path $\vp$ in $\H$, there exists $C(p,\alpha,\omega,X)$ such that
\begin{equation}\label{energies lengths differences}
|E_{p}(\vp,t)-E_{p,\delta}(\vp,t)| \leq C\delta, \quad
|\length_{p}(\vp)-\length_{p,\delta}(\vp)| \leq C\delta^{\frac{1}{p}}.
\end{equation}
Note that the constant $C$ is uniform for any path in $\H$. In application, we will choose the path to be the $\ve$-geodesic so that the path depends on the endpoints.

\begin{lma}\label{energy estimates}
For any $\vp_{0},\vp_{1}\in\H$, let $\vp^{\ve}$ be the $\ve$-geodesic connecting $\vp_{0}$ and $\vp_{1}$. Denote the corresponding energies 
by $E_{p,\delta}^{\ve}$ 
, then there exists $C(\vp_{0},\vp_{1},p,\alpha,\omega,X)$ such that for any $\delta>0$,
\begin{enumerate}\setlength{\itemsep}{1mm}
\item[(i)] $|\de_{t}E_{p,\delta}^{\ve}|\leq C\ve^{2}$;
\item[(ii)] The following inequality holds
\[
\begin{split}
E_{p,\delta}^{\ve}(t) \geq {} & \max \left\{ \int_{\{\vp_{0}>\vp_{1}\}}|\vp_{0}-\vp_{1}|^{p}
\Re(e^{-\sqrt{-1}\hat{\theta}_{0}}\Omega_{\vp_{0}}), \right.\\
& \left.\int_{\{\vp_{1}>\vp_{0}\}}|\vp_{0}-\vp_{1}|^{p}
\Re(e^{-\sqrt{-1}\hat{\theta}_{0}}\Omega_{\vp_{1}}) \right\}-C\ve^{2}.
\end{split}
\]
\end{enumerate}
\end{lma}

\begin{proof}
 Using the $\ve$-geodesic equation \eqref{ve geodesic eqn 3}, we compute
\[
\begin{split}
 \frac{dE_{p,\delta}^{\ve}}{dt}
= {} & p\int_{X}\frac{(\vp_{t}^{\ve})^{2p-1}\vp_{tt}^{\ve}}{\sqrt{(\vp_{t}^{\ve})^{2p}+\delta^{2}}}\,
\Re(e^{-\sqrt{-1}\hat{\theta}_{0}}\Omega_{\vp^{\ve}}^{n})\\
&-n\int_{X}\sqrt{(\vp_{t}^{\ve})^{2p}+\delta^{2}}\ddb\vp_{t}^{\ve}
\wedge\Im(e^{-\sqrt{-1}\hat{\theta}_{0}}\Omega_{\vp^{\ve}}^{n-1}) \\
= {} & p\int_{X}\frac{(\vp_{t}^{\ve})^{2p-1}\vp_{tt}^{\ve}}{\sqrt{(\vp_{t}^{\ve})^{2p}+\delta^{2}}}\,
\Re(e^{-\sqrt{-1}\hat{\theta}_{0}}\Omega_{\vp^{\ve}}^{n})\\
&+np\int_{X}\frac{(\vp_{t}^{\ve})^{2p-1}}{\sqrt{(\vp_{t}^{\ve})^{2}+\delta^{2}}}
\sqrt{-1}\de\vp_{t}^{\ve}\wedge\dbar\vp_{t}^{\ve}
\wedge\Im(e^{-\sqrt{-1}\hat{\theta}_{0}}\Omega_{\vp^{\ve}}^{n-1}) \\
= {} & -4pe^{-2t}\ve^{2}\int_{X}\frac{(\vp_{t}^{\ve})^{2p-1}}{\sqrt{(\vp_{t}^{\ve})^{2p}+\delta^{2}}}
\Im(e^{-\sqrt{-1}\hat{\theta}_{0}}\Omega_{\vp^{\ve}}^{n}).
\end{split}
\]
Then (i) follows from $p\geq1$ and the estimate from Theorem~\ref{ve geodesic existence and estimates}. 

\bigskip

To prove (ii). By (i) in Lemma~\ref{ve geodesic more estimates}, we have $\vp_{tt}^{\ve}\geq-C\ve^{2}$. This implies
\[
\vp_{t}^{\ve}(0) \leq \vp^{\ve}(1)-\vp^{\ve}(0)+C\ve^{2}
= \vp_{1}-\vp_{0}+C\ve^{2}.
\]
Then on the set $\{\vp_{0}>\vp_{1}\}$,
\[
|\vp_{t}^{\ve}(0)| \geq |\vp_{0}-\vp_{1}|-C\ve^{2}
\]
and so
\[
E_{p,\delta}^{\ve}(0)
\geq \int_{\{\vp_{0}>\vp_{1}\}}|\vp_{0}-\vp_{1}|^{p}
\Re(e^{-\sqrt{-1}\hat{\theta}_{0}}\Omega_{\vp_{0}}^{n})-C\ve^{2p}.
\]
Similarly, we have
\[
E_{p,\delta}^{\ve}(1)
\geq \int_{\{\vp_{1}>\vp_{0}\}}|\vp_{0}-\vp_{1}|^{p}
\Re(e^{-\sqrt{-1}\hat{\theta}_{0}}\Omega_{\vp_{1}}^{n})-C\ve^{2p}.
\]
Combining the above with (i), we obtain (ii).
\end{proof}

\begin{lma}\label{triangle inequality}
Let $\psi(s)$, $s\in[0,s_{0}]$ be a smooth path in $\H$, and $\hat{\psi}$ be a fixed point in $\H$ such that $\hat{\psi}\notin\psi([0,s_{0}])$. For any $s\in[0,s_{0}]$, let $\vp^{\ve}(s,t)$, $t\in[0,1]$ be the $\ve$-geodesic joining from $\psi(s)$ to $\hat{\psi}$. Then there exists a constant $C(\psi,\hat{\psi},p,\alpha,\omega,X)$ such that
\[
\length_{p,\delta}(\vp^{\ve}(0,\cdot))
\leq \length_{p,\delta}(\psi(\cdot))
+\length_{p,\delta}(\vp^{\ve}(s_{0},\cdot))
+C\ve^{2}\delta^{-3}.
\]
\end{lma}

\begin{proof}
Define
\[
\ell_{\delta}(s) = \length_{p,\delta}(\psi|_{[0,s]}), \quad
\hat{\ell}_{\delta}(s) = \length_{p,\delta}(\vp^{\ve}(s,\cdot)).
\]
It suffices to show that
\[
\ell_{\delta}'(s)+\hat{\ell}_{\delta}'(s) \geq -C\ve^{2}\delta^{-3}.
\]
It is clear that
\begin{equation}\label{derivative of l}
\ell_{\delta}'(s)
= E_{p,\delta}^{\frac{1}{p}}(\psi,s)
= \left(\int_{X}\sqrt{\psi_{s}^{2p}+\delta^{2}}\,
\Re(e^{-\sqrt{-1}\hat{\theta}_{0}}\Omega_{\psi}^{n})\right)^{\frac{1}{p}}
\end{equation}
and
\begin{equation}\label{derivative of hat l}
\hat{\ell}_{\delta}'(s) = \frac{1}{p}\int_{0}^{1}\big(E_{\delta}^{\ve}(s,t)\big)^{\frac{1}{p}-1}
\de_{s}E_{\delta}^{\ve}(s,t)dt,
\end{equation}
where
\[
E_{\delta}^{\ve}(s,t) = \int_{X}\sqrt{(\vp_{t}^{\ve})^{2p}+\delta^{2}}\,
\Re(e^{-\sqrt{-1}\hat{\theta}_{0}}\Omega_{\vp^{\ve}}^{n}).
\]
For notational convenience, we omit superscript $\ve$ and subscript $\delta$, i.e.,
\[
\ell = \ell_{\delta}, \ \ \hat{\ell} = \hat{\ell}_{\delta}, \ \
\vp = \vp^{\ve}, \ \ E = E_{\delta}^{\ve}.
\]
We compute
\begin{equation}\label{triangle inequality eqn 1}
\begin{split}
& \frac{\de E(s,t)}{\de s} \\
= {} & p\int_{X}\frac{\vp_{t}^{2p-1}\vp_{ts}}{\sqrt{\vp_{t}^{2p}+\delta^{2}}}
\Re(e^{-\sqrt{-1}\hat{\theta}_{0}}\Omega_{\vp}^{n})
-n\int_{X}\sqrt{\vp_{t}^{2p}+\delta^{2}}\ddb\vp_{s}
\wedge\Im(e^{-\sqrt{-1}\hat{\theta}_{0}}\Omega_{\vp}^{n-1}) \\
= {} & p\frac{\de}{\de t}\left(\int_{X}\frac{\vp_{t}^{2p-1}\vp_{s}}{\sqrt{\vp_{t}^{2p}+\delta^{2}}}
\Re(e^{-\sqrt{-1}\hat{\theta}_{0}}\Omega_{\vp}^{n})\right)
-p\int_{X}\vp_{s}\frac{\de}{\de t}\left(\frac{\vp_{t}^{2p-1}}{\sqrt{\vp_{t}^{2p}+\delta^{2}}}
\Re(e^{-\sqrt{-1}\hat{\theta}_{0}}\Omega_{\vp}^{n})\right) \\
& +np\int_{X}\frac{\vp_{t}^{2p-1}}{\sqrt{\vp_{t}^{2p}+\delta^{2}}}
\sqrt{-1}\de\vp_{t}\wedge\dbar\vp_{s}
\wedge\Im(e^{-\sqrt{-1}\hat{\theta}_{0}}\Omega_{\vp}^{n-1}).
\end{split}
\end{equation}
For the second term in \eqref{triangle inequality eqn 1}, direct calculation shows
\[
\begin{split}
& -p\int_{X}\vp_{s}\frac{\de}{\de t}\left(\frac{\vp_{t}^{2p-1}}{\sqrt{\vp_{t}^{2p}+\delta^{2}}}
\Re(e^{-\sqrt{-1}\hat{\theta}_{0}}\Omega_{\vp}^{n})\right) \\
= {} & -p\int_{X}\frac{(p-1)\vp_{t}^{4p-2}+(2p-1)\delta^{2}\vp_{t}^{2p-2}}
{(\vp_{t}^{2p}+\delta^{2})^{\frac{3}{2}}}\vp_{tt}\vp_{s}
\Re(e^{-\sqrt{-1}\hat{\theta}_{0}}\Omega_{\vp}^{n}) \\
& +np\int_{X}\frac{\vp_{t}^{2p-1}\vp_{s}}{\sqrt{\vp_{t}^{2p}+\delta^{2}}}
\ddb\vp_{t}\wedge\Im(e^{-\sqrt{-1}\hat{\theta}_{0}}\Omega_{\vp}^{n-1}).
\end{split}
\]
For the third term of \eqref{triangle inequality eqn 1}, integrating by parts yields
\[
\begin{split}
& np\int_{X}\frac{\vp_{t}^{2p-1}}{\sqrt{\vp_{t}^{2p}+\delta^{2}}}
\sqrt{-1}\de\vp_{t}\wedge\dbar\vp_{s}
\wedge\Im(e^{-\sqrt{-1}\hat{\theta}}\Omega_{\vp}^{n-1}) \\
= {} & -np\int_{X}\frac{(p-1)\vp_{t}^{4p-2}+(2p-1)\delta^{2}\vp_{t}^{2p-2}}
{(\vp_{t}^{2p}+\delta^{2})^{\frac{3}{2}}}\vp_{s}\sqrt{-1}\de\vp_{t}\wedge\dbar\vp_{t}
\wedge\Im(e^{-\sqrt{-1}\hat{\theta}_{0}}\Omega_{\vp}^{n-1}) \\
& -np\int_{X}\frac{\vp_{t}^{2p-1}\vp_{s}}{\sqrt{\vp_{t}^{2p}+\delta^{2}}}
\ddb\vp_{t}\wedge\Im(e^{-\sqrt{-1}\hat{\theta}_{0}}\Omega_{\vp}^{n-1}).
\end{split}
\]
Substituting the above into \eqref{triangle inequality eqn 1} and using the $\ve$-geodesic equation \eqref{ve geodesic eqn 3},
\[
\begin{split}
\frac{\de E(s,t)}{\de s}
= {} & p\frac{\de}{\de t}\left(\int_{X}\frac{\vp_{t}^{2p-1}\vp_{s}}{\sqrt{\vp_{t}^{2p}+\delta^{2}}}
\Re(e^{-\sqrt{-1}\hat{\theta}_{0}}\Omega_{\vp}^{n})\right) \\
& +4pe^{-2t}\ve^{2}\int_{X}\frac{(p-1)\vp_{t}^{4p-2}+(2p-1)\delta^{2}\vp_{t}^{2p-2}}
{(\vp_{t}^{2p}+\delta^{2})^{\frac{3}{2}}}\vp_{s}
\Im(e^{-\sqrt{-1}\hat{\theta}_{0}}\Omega_{\vp}^{n}).
\end{split}
\]
Applying the operator $\de_{s}$ to \eqref{ve geodesic eqn 3} and using the maximum principle, we obtain $|\vp_{s}|\leq C$. Thus, by the estimate of $\ve$-geodesic from Theorem~\ref{ve geodesic existence and estimates},
\[
\begin{split}
\frac{\de E(s,t)}{\de s}
\geq {} & p\frac{\de}{\de t}
\left(\int_{X}\frac{\vp_{t}^{2p-1}\vp_{s}}{\sqrt{\vp_{t}^{2p}+\delta^{2}}}
\Re(e^{-\sqrt{-1}\hat{\theta}_{0}}\Omega_{\vp}^{n})\right)
-C\ve^{2}\delta^{-3}.
\end{split}
\]
Substituting this into \eqref{derivative of hat l},
\[
\begin{split}
\hat{\ell}'(s) = {} & \frac{1}{p}\int_{0}^{1}E^{\frac{1}{p}-1}(\de_{s}E)dt \\
\geq {} & \left(E^{\frac{1}{p}-1}
\int_{X}\frac{\vp_{t}^{2p-1}\vp_{s}}{\sqrt{\vp_{t}^{2p}+\delta^{2}}}
\Re(e^{-\sqrt{-1}\hat{\theta}_{0}}\Omega_{\vp}^{n})\right)\Bigg|_{t=0}^{t=1} \\
& -\int_{0}^{1}(\de_{t}E^{\frac{1}{p}-1})
\int_{X}\frac{\vp_{t}^{2p-1}\vp_{s}}{\sqrt{\vp_{t}^{2p}+\delta^{2}}}
\Re(e^{-\sqrt{-1}\hat{\theta}_{0}}\Omega_{\vp}^{n})dt
-C\ve^{2}\delta^{-3}\int_{0}^{1}E^{\frac{1}{p}-1}dt.
\end{split}
\]
By the definition of $\vp$, we see that $\vp(s,0)=\psi(s)$ and $\vp(s,1)=\hat{\psi}$, which implies $\vp_{s}(s,0)=\psi_{s}(s)$ and $\vp_{s}(s,1)=0$. On the other hand, Lemma \ref{energy estimates} shows $E\geq C^{-1}$ and $|\de_{t}E^{\frac{1}{p}-1}|\leq C\ve^{2}$ provided that $\ve$ is sufficiently small. Using Theorem~\ref{ve geodesic existence and estimates} and $|\vp_{s}|\leq C$ again,
\[
\begin{split}
\hat{\ell}'(s) \geq {} & -E(s,0)^{\frac{1}{p}-1}\int_{X}\frac{\vp_{t}^{2p-1}\psi_{s}}{\sqrt{\vp_{t}^{2p}+\delta^{2}}}
\Re(e^{-\sqrt{-1}\hat{\theta}_{0}}\Omega_{\psi}^{n})-C\ve^{2}\delta^{-3}.
\end{split}
\]
By Young's inequality,
\[
\int_{X}\frac{\vp_{t}^{2p-1}\psi_{s}}{\sqrt{\vp_{t}^{2p}+\delta^{2}}}
\Re(e^{-\sqrt{-1}\hat{\theta}_{0}}\Omega_{\psi}^{n})
\leq E(s,0)^{1-\frac{1}{p}}
\left(\int_{X}|\psi_{s}|^{p}
\Re(e^{-\sqrt{-1}\hat{\theta}_{0}}\Omega_{\psi}^{n})\right)^{\frac{1}{p}}.
\]
Recalling $\ell_{\delta}'(s)=\left(\int_{X}\sqrt{\psi_{s}^{2p}+\delta^{2}}
\Re(e^{-\sqrt{-1}\hat{\theta}_{0}}\Omega_{\psi}^{n})\right)^{\frac{1}{p}}$ from \eqref{derivative of l}, we obtain
\[
\ell'(s)+\hat{\ell}'(s) \geq -C\ve^{2}\delta^{-3},
\]
as required.
\end{proof}

\begin{lma}\label{H dp lemma}
For any $\vp_{0},\vp_{1}\in\H$, let $\vp^{\ve}$ be the $\ve$-geodesic connecting $\vp_{0}$ and $\vp_{1}$. Then we have
\begin{enumerate}\setlength{\itemsep}{1mm}
\item[(i)] $d_{p}(\vp_{0},\vp_{1}) = \lim_{\ve\rightarrow0}\length_{p}(\vp^{\ve})$;
\item[(ii)] $d_{p}^{p}(\vp_{0},\vp_{1})=\lim_{\ve\rightarrow0}E_{p}^{\ve}(t)$ for any $t\in[0,1]$;
\item[(iii)] The following inequality holds
\[
\begin{split}
d_{p}^{p}(\vp_{0},\vp_{1}) \geq {} & \max \left\{ \int_{\{\vp_{0}>\vp_{1}\}}|\vp_{0}-\vp_{1}|^{p}
\Re(e^{-\sqrt{-1}\theta_{0}}(\alpha_{\vp_{0}}+\sqrt{-1}\omega)^{n}), \right.\\
& \left.\int_{\{\vp_{1}>\vp_{0}\}}|\vp_{0}-\vp_{1}|^{p}
\Re(e^{-\sqrt{-1}\theta_{0}}(\alpha_{\vp_{1}}+\sqrt{-1}\omega)^{n}) \right\}.
\end{split}
\]
\end{enumerate}
\end{lma}

\begin{proof}
(ii) and (iii) are immediate corollaries of (i), Lemma \ref{energy estimates}, \eqref{energies lengths differences} and Lemma \ref{Re Im relationships}. It remains to prove (i). By the definition of $d_{p}$,
\[
d_{p}(\vp_{0},\vp_{1}) \leq \liminf_{\ve\rightarrow0}\length_{p}(\vp^{\ve}).
\]
It suffices to prove
\[
\limsup_{\ve\rightarrow0}\length_{p}(\vp^{\ve})
\leq d_{p}(\vp_{0},\vp_{1}).
\]

Let $\psi(s)$, $s\in[0,1]$ be any smooth path connecting $\vp_{0}$ and $\vp_{1}$. We assume without loss of generality that $\vp_{1}\notin\psi([0,1))$. For each $s\in[0,1)$, let $\vp^{\ve}(s,t)$ be the $\ve$-geodesic connecting $\psi(s)$ to $\vp_{1}$. For any $s_{0}\in(0,1)$, using Lemma \ref{triangle inequality},
\[
\length_{p,\delta}(\vp^{\ve}(0,\cdot))
\leq \length_{p,\delta}(\psi|_{[0,s_{0}]})
+\length_{p,\delta}(\vp^{\ve}(s_{0},\cdot))
+C\ve^{2}\delta^{-3}.
\]
Combining this with \eqref{energies lengths differences} and $\vp^{\ve}(0,\cdot)=\vp^{\ve}$, we obtain
\[
\length_{p}(\vp^{\ve})
\leq \length_{p}(\psi|_{[0,s_{0}]})
+\length_{p}(\vp^{\ve}(s_{0},\cdot))
+C\delta^{\frac{1}{p}}+C\ve^{2}\delta^{-3}.
\]
By Lemma \ref{ve geodesic more estimates} (ii),
\[
\length_{p}(\vp^{\ve}(s_{0},\cdot))
\leq C\|\psi(s_{0})-\psi(1)\|_{L^{\infty}}+C\ve^{2}.
\]
Hence,
\[
\limsup_{\ve\rightarrow0}\length_{p}(\vp^{\ve})
\leq \length_{p}(\psi|_{[0,s_{0}]})+C\|\psi(s_{0})-\psi(1)\|_{L^{\infty}}+C\delta^{\frac{1}{p}}.
\]
By letting $\delta\rightarrow0$ and $s_{0}\rightarrow1$,
\[
\limsup_{\ve\rightarrow0}\length_{p}(\vp^{\ve})
\leq \length_{p}(\psi).
\]
Since $\psi$ is arbitrary path connecting $\vp_{0}$ and $\vp_{1}$,
\[
\limsup_{\ve\rightarrow0}\length_{p}(\vp^{\ve})
\leq d_{p}(\vp_{0},\vp_{1}),
\]
as required.
\end{proof}

Now we are in a position to prove Theorem \ref{H dp theorem}.

\begin{proof}[Proof of Theorem \ref{H dp theorem}]
The non-negativity and symmetry of $d_{p}$ are trivial.  The triangle inequality follows from Lemma \ref{triangle inequality}, \eqref{energies lengths differences} and Lemma \ref{H dp lemma} (i). The positivity of $d_{p}$ follows from Lemma \ref{H dp lemma} (iii).
\end{proof}
\bigskip

For later purpose, we will derive some inequalities on the $d_{p}$-distance.

\begin{lma}\label{dp distance inequality 1}
For any $\vp_{0},\vp_{1},\vp_{2}\in\H$ such that $\vp_{0}\leq\vp_{1}\leq\vp_{2}$, we have
\[
d_{p}^{p}(\vp_{0},\vp_{1})
\leq d_{p}^{p}(\vp_{0},\vp_{2})
\leq \int_{X}|\vp_{0}-\vp_{2}|^{p}
\Re(e^{-\sqrt{-1}\theta_{0}}(\alpha_{\vp_{0}}+\sqrt{-1}\omega)^{n}).
\]
\end{lma}

\begin{proof}
For $i=0,1,2$, let $\vp^{\ve,i}$ and $\Phi^{\ve,i}$ be the $\ve$-geodesic joining from $\vp_{0}$ to $\vp_{i}$. Denote the corresponding $L^{p}$-energy by $E_{p}^{\ve,i}$. By comparison principle,
\[
\Phi^{\ve,0} \leq \Phi^{\ve,1} \leq \Phi^{\ve,2}
\]
and so
\[
\vp^{\ve,0} \leq \vp^{\ve,1} \leq \vp^{\ve,2}.
\]
This implies
\[
\vp^{\ve,0}_{t}(0) \leq \vp^{\ve,1}_{t}(0) \leq \vp^{\ve,2}_{t}(0).
\]
Thanks to Lemma \ref{ve geodesic more estimates},
\[
\vp^{\ve,2}_{t}(0) \leq \vp_{2}-\vp_{0}+C\ve^{2}, \quad
|\vp^{\ve,0}_{t}(0)| \leq C\ve^{2}.
\]
It then follows that
\[
-C\ve^{2} \leq \vp^{\ve,0}_{t}(0)
\leq \vp^{\ve,1}_{t}(0) \leq \vp^{\ve,2}_{t}(0)
\leq \vp_{2}-\vp_{0}+C\ve^{2}
\]
and
\[
E_{p}^{\ve,1}(0) \leq E_{p}^{\ve,2}(0)+C\ve^{2p}
\leq \int_{X}|\vp_{0}-\vp_{2}|^{p}
\Re(e^{-\sqrt{-1}\hat{\theta}_{0}}\Omega_{\vp_{0}}^{n})+C\ve^{2p}.
\]
Letting $\ve\rightarrow0$ and using Lemma \ref{H dp lemma} (ii), we obtain
\[
d_{p}^{p}(\vp_{0},\vp_{1})
\leq d_{p}^{p}(\vp_{0},\vp_{2})
\leq \int_{X}|\vp_{0}-\vp_{2}|^{p}\Re(e^{-\sqrt{-1}\hat{\theta}_{0}}\Omega_{\vp_{0}}^{n}).
\]
Recalling the definition of $\Omega_{\vp_{0}}$ (see \eqref{Omega vp}) and using Lemma \ref{Re Im relationships},
\[
\Re(e^{-\sqrt{-1}\hat{\theta}_{0}}\Omega_{\vp_{0}}^{n})
= \Re(e^{-\sqrt{-1}\hat{\theta}_{0}}(\omega+\sqrt{-1}\alpha_{\vp_{0}})^{n})
= \Re(e^{-\sqrt{-1}\theta_{0}}(\alpha_{\vp_{0}}+\sqrt{-1}\omega)^{n}).
\]
Then we obtain the required inequality.
\end{proof}

\begin{lma}\label{dp distance inequality 2}
For $c_{0}\in(0,\Theta_{0})$, there exists $C(c_{0},\alpha,\omega,X)$ such that
for any $\vp\in\H$ with $\sup_{X}\vp=0$ and $Q_{\omega}(\alpha_{\vp})\leq\Theta_{0}-c_{0}$,
\[
d_{p}^{p}(\vp,0) \geq \frac{1}{C}\int_{X}|\vp|^{p}\chi_{\vp}^{n}.
\]
where $\Theta_{0}=\theta_{0}+\frac{\pi}{2}$ and $\chi,\chi_{\vp}$ are defined in Lemma \ref{quasi-plurisubharmonic}.
\end{lma}

\begin{proof}
By Lemma \ref{dp distance inequality 1} and \ref{H dp lemma} (iii),
\begin{equation}\label{dp distance inequality 2 eqn 1}
\begin{split}
d_{p}^{p}(\vp,0) \geq {} & d_{p}^{p}\left(\vp,\frac{\vp}{2}\right)
\geq \frac{1}{2^{p}}\int_{X}|\vp|^{p}
\Re(e^{-\sqrt{-1}\hat{\theta}_{0}}\Omega_{\frac{\vp}{2}}^{n}) \\
= {} & \frac{1}{2^{p}}\int_{X}|\vp|^{p}
\Re\left(e^{-\sqrt{-1}\theta_{0}}(\alpha_{\frac{\vp}{2}}+\sqrt{-1}\omega)^{n}\right),
\end{split}
\end{equation}
where we used Lemma \ref{Re Im relationships} in the last equality.

Let $\lambda_{i}$ and $\mu_{i}$ be the eigenvalues of $\alpha_{\vp}$ and $\alpha_{\frac{\vp}{2}}$ with respect to $\omega$. It is clear that
\[
\alpha_{\frac{\vp}{2}} = \frac{1}{2}\alpha+\frac{1}{2}\alpha_{\vp}.
\]
Thanks to the concavity of $\cot(Q_{\omega})$, $Q_{\omega}(\alpha)\leq\Theta_{0}-c_{0}'$ and $Q_{\omega}(\alpha_{\vp})\leq\Theta_{0}-c_{0}$,
\[
Q_{\omega}(\alpha_{\frac{\vp}{2}})\leq\Theta_{0}-c_{0}''.
\]
On the other hand, by Weyl's inequality, for each $i$,
\[
\left|\mu_{i}-\frac{\lambda_{i}}{2}\right| \leq C.
\]
Then we have
\[
\frac{\Re(e^{-\sqrt{-1}\theta_{0}}(\alpha_{\frac{\vp}{2}}+\sqrt{-1}\omega)^{n})}{\omega^{n}}
\geq \cos(Q_{\omega}(\alpha_{\frac{\vp}{2}})-\theta_{0})\prod_{i=1}^{n}\sqrt{1+\mu_{i}^{2}}
\geq \frac{1}{C}\prod_{i=1}^{n}\sqrt{1+\lambda_{i}^{2}}.
\]
Recalling $\chi_{\vp}=\tan(\theta_{0})\omega+\alpha_{\vp}$,
\[
\frac{\chi_{\vp}^{n}}{\omega^{n}}
= \prod_{i=1}^{n}(\tan(\theta_{0})+\lambda_{i})
\leq C\prod_{i=1}^{n}\sqrt{1+\lambda_{i}^{2}}
\leq C\cdot\frac{\Re(e^{-\sqrt{-1}\theta_{0}}(\alpha_{\frac{\vp}{2}}+\sqrt{-1}\omega)^{n})}{\omega^{n}}.
\]
Substituting this into \eqref{dp distance inequality 2 eqn 1}, we obtain
\[
d_{p}^{p}(\vp,0) \geq \frac{1}{C}\int_{X}|\vp|^{p}\chi_{\vp}^{n},
\]
as required.
\end{proof}

\section{Twisted dHYM flow}\label{Section:dHYM flow}

In this section, we will study the twisted dHYM flow starting from $\varphi_0\in \H$:
\begin{equation}\label{twisted dHYM flow 2}
\left\{
\begin{array}{ll}
\partial_t \varphi&=F_{\omega,\ve}(\a_\varphi)-\cot\theta_0-a_0\ve ;\\[1mm]
\varphi(0)&=\varphi_0.
\end{array}
\right.
\end{equation}
Recall that the twisted dHYM flow admits an unique short-time solution on $X$ thanks to Lemma~\ref{F properties} and \eqref{derivatives of F}. In fact, the linearised operator is given by $$\Box_L=\partial_t- F^{i\bar j}_{\omega,\ve}\nabla_i \nabla_{\bar j}$$ which is parabolic as long as the eigenvalues of $\a_\varphi$ with respect to $\omega$ lies inside  $\Gamma_{\theta,\Theta}$. We let $T_{\max}>0$ be the maximal existence time of \eqref{twisted dHYM flow 2}. In this section, we will study its long-time existence and behaviour. We first show that the $\Box_L$ will remain parabolic as long as the flow exists.

\begin{lma}\label{lemma:parabolic-of-L}
There are $\ve_0(\varphi_0,\a,\omega,X), c_0(\varphi_0,\a,\omega,X)>0$ such that if $\ve<\ve_0$, then for all $t\in [0,T_{\max})$, 
$$c_0\leq Q_\omega(\a_\varphi)\leq \Theta_0-c_0.$$
In particular, $\a_\varphi\in \Gamma_{\omega,\Theta_{0}-c_0,\Theta_{0}}$ and $\a_\varphi>-(\tan\theta_0) \omega$ for all $t\in [0,T_{\max})$.
\end{lma}
\begin{proof}

Let $c_0>0$ be a constant to be chosen. Define
\[
S = \sup\{\,s>0~|~\text{$Q_\omega(\a_\varphi)\in [c_0,\Theta_0-c_0]$ for $t<s$}\,\}.
\]
Since $\varphi_0\in \H$, when $c_0$ is sufficiently small, the set on the right hand side is not empty and so the number $S$ is well-defined. If $S=T_{\max}$, then we are done. It suffices to rule out the case $S<T_{\max}$. By differentiating \eqref{twisted dHYM flow 2} with respect to $t$,
\[
\Box_L F_{\omega,\ve}(\a_{\varphi})=0.
\]
By maximum principle and passing $t\to S$, we conclude that for all $t\in [0,S]$,
\[
\inf_X F_{\omega,\ve}(\a_{\varphi_0})\leq F_{\omega,\ve}(\a_{\varphi})\leq \sup_X F_{\omega,\ve}(\a_{\varphi_0}).
\]
Since $\varphi_0\in \H$, $c_1\leq Q_\omega(\a_{\varphi_0})\leq \Theta_0-c_1$ for some $c_1(\varphi_0,\alpha,\omega,X)>0$ and hence
\[
\cot(\Theta_0-c_1)\leq  F_{\omega,\ve}(\a_{\varphi_0}) \leq\cot(c_1)+\ve \csc(c_1).
\]
Therefore,
\[
\begin{split}
\cot (Q_\omega(\a_\varphi))&\leq F_{\omega,\ve}(\a_\varphi)
\leq \cot(c_1)+\ve \csc(c_1) = \cot(c_{2}),
\end{split}
\]
which shows $Q_\omega(\a_\varphi)\geq c_2$ for some $c_{2}(\vp_{0},\alpha,\omega,X)>0$ on $[0,S]$.

Next, we will show $Q_\omega(\a_\varphi)<\Theta_{0}-\frac{c_1}{2}$ on $[0,S]$. If $Q_\omega(\a_\varphi) \geq \Theta_0-\frac{c_1}{2}$ at some point, then by
\[
\cot(Q_\omega(\a_\varphi))+\frac{\ve \omega^n}{\im (\a_\varphi+\sqrt{-1}\omega)^n}=F_{\omega,\ve}(\a_\varphi)
\geq \inf_X F_{\omega,\ve}(\a_{\varphi_0})
\geq \cot(\Theta_0-c_1),
\]
we obtain
\[
\begin{split}
\frac{\ve}{\sin(Q_{\omega}(\alpha_{\vp}))} \geq
\frac{\ve \omega^n}{\im (\a_\varphi+\sqrt{-1}\omega)^n}\geq  \cot(\Theta_0-c_1)-\cot\left(\Theta_0-\frac{c_1}{2}\right).
\end{split}
\]
Hence,
\[
Q_\omega(\a_\varphi)\leq c_3 \ve
\]
for some $c_3(\vp_{0},\alpha,\omega,X)>0$, which contradicts with $Q_\omega(\a_\varphi) \geq \Theta_0-\frac{c_1}{2}$ after decreasing $\ve_{0}$ if necessary.

We now have $c_2\leq Q_\omega(\a_\varphi)<\Theta_{0}-\frac{c_1}{2}$ on $[0,S]$. Choose $c_0=\min(\frac{c_1}{2},c_2)$, then $Q_\omega(\a_\varphi)\in (c_0,\Theta_0-c_0)$ for all $t\in [0,S]$, which contradicts with the maximality of $S$. This shows $S<T_{\max}$ is impossible, or equivalently, $S=T_{\max}$. Since $P_\omega(\a_\varphi)\leq Q_\omega(\a_\varphi)$, then $\a_\varphi\in \Gamma_{\omega,\Theta_{0}-c_0,\Theta_{0}}$. The lower bound of $\a_\varphi$ follows from Lemma~\ref{quasi-plurisubharmonic}.
\end{proof}

It is known that the twisted dHYM flow is the negative gradient flow of $\J_\ve$-functional which is basically from construction. {For the later purpose, we will consider the $Z$-functional defined in Collins-Yau \cite{CollinsYau2021} which is defined to be $Z=e^{-\sqrt{-1}\frac{n\pi}{2}}\mathrm{CY}_\C$ in term of the Calabi-Yau functional. Or equivalently, it is given by $Z(0)=0$ and
\[
\delta Z(u)=\int_X (\delta u) e^{-\sqrt{-1}\frac{n\pi}{2}}\left(\omega+\sqrt{-1}\alpha_{u}\right)^n.
\]
By the similar calculation of Lemma \ref{Re Im relationships}, we obtain
\[
\Im\left(e^{-\sqrt{-1}\frac{n\pi}{2}}\left(\omega+\sqrt{-1}\alpha_{u}\right)^n\right)
= -\Im\left(\alpha_{u}+\sqrt{-1}\omega\right)^n
\]
which gives us
\[
\delta (\Im Z)(u) = -\int_X (\delta u)\Im\left(\alpha_{u}+\sqrt{-1}\omega\right)^n.
\]
}


\begin{lma}\label{CY-const}
Along the twisted dHYM flow \eqref{twisted dHYM flow 2}, we have  for all $t\in [0,T_{\max})$, $$\im (Z(\varphi(t)))= \im (Z(\varphi_0)).$$
\end{lma}
\begin{proof}
Direct calculation shows
\begin{equation}\label{CY-const eqn}
\begin{split}
\frac{d}{dt} \mathrm{Im}\left( Z(\varphi)\right)
&=-\int_X \left(F_{\omega,\ve}(\a_\varphi)-\cot\theta_0-a_0\ve \right) \im\left(\a_{\varphi}+\sqrt{-1}\omega \right)^n\\
&=-\int_X \re \left(\a_{\varphi}+\sqrt{-1}\omega \right)^n-\cot\theta_0\cdot \im\left(\a_{\varphi}+\sqrt{-1}\omega \right)^n \\
&\quad - \int_X \ve\omega^n-a_0 \ve \im\left(\a_{\varphi}+\sqrt{-1}\omega \right)^n \\
&=0.
\end{split}
\end{equation}
\end{proof}

We also need the following evolution equation along the twisted dHYM flow.
\begin{lma}\label{lemma-evolution-alpha}
The twisted dHYM flow \eqref{twisted dHYM flow 2} satisfies
\[
\begin{split}
\Box_L (\a_\varphi)_{i\bar j}&=F^{p\bar q,k\bar l}_{\omega,\ve}\nabla_{i} \a_{p\bar q} \nabla_{\bar j} \a_{k\bar l}+F^{p\bar q}_{\omega,\ve}\left(R_{p\bar ji}\,^k \, \a_{k\bar q}-R_{p\bar j}\,^{\bar l}_{\bar q}\,\a_{i\bar l}\right).
\end{split}
\]
\end{lma}
\begin{proof}
In the following, all connections are computed with respect to $\omega$. We will omit the subscript for notational convenience. We compute using $d$-closed of $\a$ and Ricci identity to obtain:
\[
\begin{split}
\partial_t \a_{i\bar j}&=\partial_t\partial_i\partial_{\bar j} \varphi\\
&=\partial_i\partial_{\bar j}F_{\omega,\ve}\\
&=\nabla_{\bar j} \left(F^{p\bar q}_{\omega,\ve} \nabla_{i} \a_{p\bar q}\right)\\
&=F^{p\bar q,k\bar l}_{\omega,\ve}\nabla_{i} \a_{p\bar q} \nabla_{\bar j} \a_{k\bar l}+F^{p\bar q}_{\omega,\ve} \nabla_{\bar j} \nabla_i\a_{p\bar q}\\
&=F^{p\bar q,k\bar l}_{\omega,\ve}\nabla_{i} \a_{p\bar q} \nabla_{\bar j} \a_{k\bar l}+F^{p\bar q}_{\omega,\ve} \nabla_{\bar j}\nabla_{p} \a_{i\bar q}\\
&=F^{p\bar q,k\bar l}_{\omega,\ve}\nabla_{i} \a_{p\bar q} \nabla_{\bar j} \a_{k\bar l}+F^{p\bar q}_{\omega,\ve}\left(\nabla_{p} \nabla_{\bar j} \a_{i\bar q} +R_{p\bar ji}\,^k \, \a_{k\bar q}-R_{p\bar j}\,^{\bar l}_{\bar q}\,\a_{i\bar l}\right)\\
&=F^{p\bar q,k\bar l}_{\omega,\ve}\nabla_{i} \a_{p\bar q} \nabla_{\bar j} \a_{k\bar l}+F^{p\bar q}_{\omega,\ve}\left(\nabla_{p} \nabla_{\bar q} \a_{i\bar j} +R_{p\bar ji}\,^k \, \a_{k\bar q}-R_{p\bar j}\,^{\bar l}_{\bar q}\,\a_{i\bar l}\right).
\end{split}
\]
%

\end{proof}

\subsection{Long-time existence}
In this subsection, we will show that if $\ve$ is sufficiently small, the corresponding twisted dHYM flow \eqref{twisted dHYM flow 2} will exist for all $t>0$.

\begin{thm}\label{long time existence}
There is $\ve_0(\varphi_0,\alpha,\omega,X)>0$ such that for all $\ve<\ve_0$, the  twisted dHYM flow \eqref{twisted dHYM flow 2} admits a unique solution on $X\times [0,+\infty)$.
\end{thm}
\begin{proof}
Suppose on the contrary, $T_{\max}<+\infty$. Since $\dot\varphi$ is uniformly bounded along \eqref{twisted dHYM flow 2} by Lemma~\ref{lemma:parabolic-of-L}, $\varphi$ is uniformly bounded in finite time. Moreover, Lemma~\ref{lemma:parabolic-of-L} implies $\varphi\in \H$ for $t\in [0,T_{\max})$. Since $F_{\omega,\ve}$ is concave, by standard parabolic theory, if $|\ddb\varphi|$ is uniformly bounded on $[0,T_{\max})$, then $\varphi$ is bounded in $C^k$ for all $k$ on $[0,T_{\max})$. This will contradict the maximality of $T_{\max}<+\infty$. Thanks to the lower bound of $\a_\varphi$ from Lemma~\ref{lemma:parabolic-of-L}, it suffices to estimate $\tr_\omega\a_\varphi$. In the following, we will use $C_i$ to denote constants depending only on $\varphi_0,\omega,\a,X$ and omit subscript $\varphi,\ve$ for notational convenience.

Denote the associated Riemannian metric of $\omega$ by $g$. Taking trace on the equation from Lemma~\ref{lemma-evolution-alpha}, we have
\[
\begin{split}
 \Box_L \log\tr_\omega \a
&\leq  \frac1{\tr_\omega \a}g^{i\bar j}F^{p\bar q,k\bar l} \nabla_{\bar j} (\a_\varphi)_{p\bar q} \nabla_{i} (\a_\varphi)_{k\bar l}\\
&\quad +\frac1{(\tr_\omega \a)^2} g^{p\bar q}g^{k\bar l} F^{i\bar j} \nabla_i \a_{p\bar q}\nabla_{\bar j}\a_{k\bar l}+\frac{C_n}{\tr_\omega \a_\varphi} |F^{i\bar j}| |\Rm| |\a|.
\end{split}
\]

At each $(x,t)$, we will work on the coordinate so that
$$g_{i\bar j}=\delta_{ij},\quad  (\a_\varphi)_{i\bar j}=\lambda_i\delta_{ij}.$$
Hence using \eqref{derivatives of F} and the concavity from Lemma~\ref{F properties},  we conclude that
\[
\begin{split}
\Box_L \log\tr_\omega \a_\varphi
&\leq \frac{ F^{i\bar i} }{(\tr_\omega \a)^2}|\nabla_i \log \tr_\omega\a_\varphi|^2+C_0 \sum_{i=1}^n F_{i}.
\end{split}
\]

Recall that
$$
F_\ve(\lambda)=\cot Q(\lambda)+\frac{\ve\csc Q(\lambda)}{\prod_{j=1}^n \sqrt{\lambda_j^2+1} }.
$$
Together with the bound of $Q(\lambda)$ from Lemma~\ref{lemma:parabolic-of-L},
\begin{equation}\label{Fii-bound}
F^{i\bar i}=F_i=\frac{\partial F_\ve}{\partial \lambda_i}
=\frac{\csc^2 Q}{1+\lambda_i^2}+\frac{\ve \csc Q \left( { \cot Q}-{\lambda_i}\right)}{(\lambda_i^2+1)\cdot \prod_{j=1}^n \sqrt{\lambda_j^2+1} }\leq C_1.
\end{equation}
Therefore, the function $G=\log\tr_\omega\a_\varphi-2C_2t$ satisfies
\[
\Box_L \log\tr_\omega \a_\varphi
\leq \frac{ F^{i\bar i} }{(\tr_\omega \a)^2}|\nabla_i \log \tr_\omega\a_\varphi|^2-C_2.
\]

For any $S<T_{\max}$, if $G$ attains its maximum at $(x_0,t_0)\in X\times [0,S]$ where $t_0>0$, then at this point,
$$0\leq \Box_LG \leq -C_2$$
which is impossible. Here we have used the fact that $\nabla G|_{(x_0,t_0)}=0$. By passing $S\to T_{\max}$, we have shown that for all $t\in [0,T_{\max})$,
\[
\log\tr_\omega \a_\varphi \leq C_3 (t+1) \leq C_3 (T_{\max}+1).
\]
This completes the proof.
\end{proof}

\subsection{Long-time behaviour}


In this section, we will study the convergence of twisted dHYM flow assuming the existence
of a $\mathcal{C}$-subsolution. The following is the main result of this subsection which is a twisted version of result in \cite{FuZhang2021}.
\begin{thm}\label{smooth convergence}
Suppose that the dHYM equation \eqref{intro-dHYM} admits a solution $\hat{\vp}_{0}$. There exists $\ve_{0}(\hat{\vp}_{0},\alpha,\omega,X)$ such that for any $\ve\in(0,\ve_{0})$, the twisted dHYM flow \eqref{twisted dHYM flow 2} will converge smoothly to $\hat{\vp}_{\ve}$ as $t\to+\infty$, where $\hat{\vp}_{\ve}$ is the solution of twisted dHYM equation, i.e.,
\[
F_{\omega,\ve}(\alpha_{\hat{\vp}_{\ve}}) = \cot(\theta_{0})+a_{0}\ve.
\]
\end{thm}


By Theorem~\ref{long time existence}, it is known that the twisted dHYM flow \eqref{twisted dHYM flow 2} exists for all $t\geq 0$. It remains to study its convergence. Before we prove the  uniform $C^k$ estimate along the flow, we first recall the notion of $\mathcal{C}$-subsolution from \cite{PhongTo2017} in the content of parabolic twisted dHYM equation.
\begin{defn}
A smooth function $\underline{u}$ is said to be a (parabolic) $\mathcal{C}$-subsolution of \eqref{twisted dHYM flow 2}, if there exist constants $\delta,K>0$, so that for any $(z,t)\in X\times [0,T)$, the condition
$$(\star) \quad F\left(\lambda(\alpha_{\underline{u}})+\mu \right)-\partial_t \underline{u}+\tau= \cot\theta_0 +a_0\ve,\quad\;\mu+\delta I_n \in \Gamma_n,\quad\tau>-\delta $$
implies that $|\mu|+|\tau|\leq K$, where $\lambda(\alpha_{\underline{u}})$ denotes the eigenvalue of $\alpha_{\underline{u}}$ with respect to $\omega$, and $I_n$ denotes the vector $(1,...,1)$.
\end{defn}

By the work of \cite{CollinsJacobYau2020}, it is now known that existence of subsolution is
equivalent to the existence of dHYM equation. The following Lemma shows
that a dHYM equation is indeed a parabolic $C$-subsolution in the sense of \cite{PhongTo2017}
as expected, see also \cite[Section 5]{Takahashi2020}.

\begin{lma}\label{para-Csub}
Suppose the dHYM equation \eqref{intro-dHYM} admits a solution $\hat\varphi_0$, then there is $\ve_0(\hat\varphi_0,\a,\omega,X)>0$ such that for all $\ve<\ve_0$, $\hat\varphi_0$ is a  (parabolic) $C$-subsolution of \eqref{twisted dHYM flow 2}.
\end{lma}
\begin{proof}
Note that $\underline{u}=\hat\varphi_0$ is time independent and so $\partial_t \underline{u}\equiv 0$. { Combining this with $(\star)$ and the monotonicity of $F$ from Lemma \ref{F properties} (i), we see that
\[
\tau = \cot\theta_0 +a_0\ve-F\left(\lambda(\alpha_{\underline{u}})+\mu \right)
\leq \cot\theta_0 +a_0\ve-F\left(\lambda(\alpha_{\underline{u}})-\delta I_n \right).
\]
}
It remains to consider $\mu$.
The equation $(\star)$ can be rewritten as
\[
\begin{split}
\cot Q+\frac{\ve}{\sin Q\cdot  \prod_{i=1}^n \sqrt{(\lambda_i+\mu_i)^2+1}} =\cot \theta_0+a_0\ve-\tau,
\end{split}
\]
where $Q=\sum_{i=1}^n \arccot (\lambda_i+\mu_i)$. Suppose $\mu_j\geq \Lambda$ for some $j$, where $\Lambda$ is a large constant to be specified. Then
\[
\begin{split}
\sum_{i\neq j} \arccot (\lambda_i-\delta)
&\geq \sum_{i\neq j} \arccot (\lambda_i+\mu_i)\\[1mm]
&\geq Q-\arccot(\lambda_j+\Lambda)\\[5mm]
&\geq \arccot \left(\cot \theta_0+a_0\ve +\delta\right)-\arccot(\lambda_j+\Lambda)\\[1mm]
&=  \arccot \left[\cot \left(\sum_{i=1}^n \arccot(\lambda_i) \right)+a_0\ve +\delta\right]-\arccot(\lambda_j+\Lambda).
\end{split}
\]
If we choose $\Lambda$ sufficiently large and $\ve_0,\delta$ sufficiently small,  it will be impossible and hence $\mu_{j}<\Lambda$ for all $j$. This gives the upper bound of $\mu$.
\end{proof}

\begin{lma}\label{0rd-order-longtime}
There exist $C_0,\ve_0>0$ depending only on $\varphi_0,\hat\varphi_0,\a,\omega,X$ such that for all $\ve<\ve_0$, the twisted dHYM flow \eqref{twisted dHYM flow 2} satisfies
$$\sup_{X\times [0,+\infty)}|\varphi|\leq C_0.$$
\end{lma}
\begin{proof}
By Lemma~\ref{CY-const}, we have
\begin{equation}
\label{cont-CY-flow} \int_X (\varphi-\varphi_0) \int^1_0 \im \left( \a_{s\varphi+(1-s)\varphi_0}+\sqrt{-1}\omega\right)^n ds=0.
\end{equation}
and $\varphi$ is uniformly quasi-plurisubharmonic for all $t\geq 0$ thanks to Lemma~\ref{lemma:parabolic-of-L}. Since $\hat\varphi_0$ serves as a $\mathcal{C}$-subsolution by Lemma~\ref{para-Csub} and the operator $F_{\omega,\ve}$ is elliptic by Lemma~\ref{F properties} (i), the proof of \cite[Lemma 3.4, 3.5]{FuZhang2021} which is based on the work \cite{PhongTo2017} can now be carried over, see also \cite[Lemma 5.3]{Takahashi2020}.
\end{proof}

It is known that the solution to dHYM equation naturally provides a $\mathcal{C}$-subsolution to the twisted dHYM equation. This will provide us a barrier to establish higher order estimates along \eqref{twisted dHYM flow 2}. We first need the following key properties of $\mathcal{C}$-subsolutions.
\begin{lma}\label{sub-sol-TdHYMFlow}
There exist $\rho,K,\ve_0>0$ depending only on $\varphi_0,\hat\varphi_0,\a,\omega,X$ such that for $\ve<\ve_0$, if $|\a_\varphi-\a_{\hat \varphi_0}|>K$, then either
$$\Box_L (\varphi-\hat\varphi_0)> \rho\sum_{i=1}^n F_{\omega,\ve}^{i\bar i}$$
or for all $1\leq i\leq n$,
$$F_{\omega,\ve}^{i\bar i}> \rho\sum_{j=1}^n F_{\omega,\ve}^{j\bar j}.$$
\end{lma}
\begin{proof}
Since $\hat\varphi_0$ is a $\mathcal{C}$-subsolution by Lemma~\ref{para-Csub} and $F_{\omega,\ve}$ is concave from Lemma~\ref{F properties} (iii). The argument in \cite[Lemma 3]{PhongTo2017} can now be carried over, see also \cite[Lemma 5.4]{Takahashi2020}.
\end{proof}
Before we prove the second order estimate, we first show that the first order is controlling the second order in an improved rate.
\begin{lma}\label{2nd-order-longtime}
There exist $C,\ve_0>0$ depending only on $\varphi_0,\hat\varphi_0,\a,\omega,X$ such that for all $T\in (0,+\infty)$, $\ve<\ve_0$,
$$\sup_{X\times [0,T]}|\ddb\varphi|\leq C(1+\sup_{X\times [0,T]} |\nabla\varphi|).$$
\end{lma}
\begin{proof}
We will always assume $\ve_0$ is sufficiently small so that the previous lemmas hold. Lemma~\ref{lemma:parabolic-of-L} shows $\a_\varphi>-(\tan\theta_0)\omega$. Then it suffices to estimate the upper bound of $\a_\varphi$. We will follow the argument in \cite{FuZhang2021} where the test function only involves the zeroth and second order quantities. In the following, we will use $C_i$ to denote any constants depending only on $n,\varphi_0,\hat\varphi_0,\a,\omega,X$. We will also omit the subscript on $\a_\varphi, F_{\omega,\ve}$ for notational convenience.

On any compact time interval $[0,T]$, consider the function
$$G=\log \lambda_{\max}+\phi(u),$$
where $\lambda_{\max}$ is the largest eigenvalue of $\a_\varphi$ with respect to $\omega$, $u=\varphi-\hat\varphi_0$ and $\phi$ is non-increasing function on $\mathbb{R}$ which will be chosen later. Note that $u$ is uniformly bounded by Lemma~\ref{0rd-order-longtime}, and hence $\phi$ can be chosen so that its derivatives are uniformly bounded for all time.

On $[0,T]$, suppose $G$ attains its maximum at $(x_0,t_0)\in X\times (0,T]$. We note that although $G$ is only continuous in general, we will use the perturbation technique as in \cite[Section 4]{Szekelyhidi2018}. We therefore will assume $G$ to be differentiable when we apply maximum principle. We will choose a coordinate at $(x_0,t_0)$ so that
$$g_{i\bar j}=\delta_{ij} \quad \text{and}\quad (\a_\varphi)_{i\bar j}=\lambda_i\delta_{ij}$$
with $\lambda_1\geq  ...\geq \lambda_n$ so that $\lambda_1=\lambda_{\max}$. As in \cite{Szekelyhidi2018}, we may assume $\lambda_1>\lambda_2$. Furthermore, we may assume $\lambda_1\geq 1$, otherwise the result holds trivially. Recall the formulas of the variation of the largest eigenvalue:
\[
\left\{
\begin{array}{ll}
\nabla_i \lambda_1=\nabla_i \a_{1\bar 1}\\
\nabla_i\nabla_{\bar i} \lambda_1=\displaystyle\nabla_i \nabla_{\bar i} \a_{1\bar 1}+\sum_{k=2}^n \frac{|\nabla_i \a_{1\bar k}|^2+|\nabla_{\bar i} \a_{k\bar 1}|^2}{\lambda_1-\lambda_k}.
\end{array}
\right.
\]
Combining this with Lemma \ref{lemma-evolution-alpha}, we compute
\[
\begin{split}
\Box_L \log \lambda_{1} &=\lambda_1^{-1} \Box_L \lambda_{1}
+\lambda_1^{-2}F^{i\bar i} |\nabla_i \lambda_1|^2\\
&=\lambda_1^{-1} \Box_L \a_{1\bar 1} +\lambda_1^{-2}F^{i\bar i} |\nabla_i \lambda_1|^2-\lambda_1^{-1} F^{i\bar i} \sum_{k=2}^n \frac{|\nabla_i \a_{1\bar k}|^2+|\nabla_{\bar i} \a_{k\bar 1}|^2}{\lambda_1-\lambda_k}\\
&= \lambda_1^{-1} \left[F^{p\bar q,k\bar l}\nabla_{1} \a_{p\bar q} \nabla_{\bar 1} \a_{k\bar l}+F^{p\bar q}\left(R_{p\bar ji}\,^k \, \a_{k\bar q}-R_{p\bar j}\,^{\bar l}_{\bar q}\,\a_{i\bar l}\right) \right]\\
&\quad +\lambda_1^{-2}F^{i\bar i} |\nabla_i \lambda_1|^2-\lambda_1^{-1} F^{i\bar i} \sum_{k=2}^n \frac{|\nabla_i \a_{1\bar k}|^2+|\nabla_{\bar i} \a_{k\bar 1}|^2}{\lambda_1-\lambda_k}.
\end{split}
\]
Using \eqref{derivatives of F} and \eqref{Fii-bound}, we see that
\begin{equation}\label{2nd-order-longtime eqn 1}
\begin{split}
\Box_L \log \lambda_{1}
&\leq  \lambda_1^{-1}F_{pk}|\nabla_1 \a_{i\bar i}|^2+\lambda_1^{-1}\sum_{p\neq q}\frac{F_p-F_q}{\lambda_p-\lambda_q}|\nabla_1 \a_{p\bar q}|^2\\
&\quad -\lambda_1^{-1} F_i \sum_{k=2}^n \frac{|\nabla_i \a_{1\bar k}|^2+|\nabla_{\bar i} \a_{k\bar 1}|^2}{\lambda_1-\lambda_k}+\lambda_1^{-2}F_i |\nabla_i \lambda_1|^2+C_0 \\
&\leq \lambda_1^{-1}\sum_{i=2}^n\frac{F_i-F_1}{\lambda_i-\lambda_1}|\nabla_i\lambda_1|^2 +\lambda_1^{-2}\sum_{i=1}^nF_i |\nabla_i \lambda_1|^2+C_0\\
&= \lambda_1^{-2}\sum_{i=2}^n\left(\frac{\lambda_1(F_i-F_1)}{\lambda_i-\lambda_1} +F_i\right)|\nabla_i\lambda_1|^2 +\lambda_1^{-2}F_1 |\nabla_1 \lambda_1|^2+C_0,
\end{split}
\end{equation}
where we have used concavity of $F$ and $d\a=0$ in the second-to-last line. To deal with the first term on the right hand side, we use \eqref{Fii-bound} again and compute
\begin{equation}\label{2nd-order-longtime eqn 2}
\begin{split}
&\quad \frac{\lambda_1(F_i-F_1)}{\lambda_i-\lambda_1}+F_i\\
&= -\frac{\csc^2 Q}{1+\lambda_i^2}\frac{\lambda_1^2+\lambda_1\lambda_i}{1+\lambda_1^2} +\frac{\ve \lambda_1 \csc Q}{\prod_{j=1}^n \sqrt{\lambda_j^2+1}}\left(\frac{-(\lambda_1+\lambda_i)\cot Q+\lambda_1\lambda_i-1}{(1+\lambda_1^2)(1+\lambda_i^2)} \right)\\
&\quad + \frac{\csc^2 Q }{1+\lambda_i^2} +\frac{\ve \csc Q (\cot Q-\lambda_i)}{(\lambda_i^2+1)\cdot\prod_{j=1}^n \sqrt{\lambda_j^2+1}}\\
&=\frac{\csc^2 Q}{1+\lambda_i^2} \cdot\frac{1-\lambda_1 \lambda_i}{1+\lambda_1^2}\\
&\quad +\frac{\ve \csc Q}{(\lambda_i^2+1)\cdot\prod_{j=1}^n\sqrt{\lambda_j^2+1}} \left[ \frac{ \cot Q (1-\lambda_1\lambda_i)-(\lambda_1+\lambda_i)}{1+\lambda_1^2}\right] \\
&=\frac{1-\lambda_1 \lambda_i}{1+\lambda_1^2} F_i - \frac{\ve \csc Q }{\prod_{j=1}^n \sqrt{\lambda_j^2+1}}\cdot \frac{\lambda_1}{1+\lambda_1^2}.
\end{split}
\end{equation}

We now squeeze more negativity from the first term. We will use an observation which is originated in \cite{WangYuan2014,Yuan2006}. Lemma \ref{lemma:parabolic-of-L} shows $Q\leq \Theta_0-c_0$. Then we see that $\lambda_{n-1}\geq C_1^{-1}$ since otherwise $Q\geq \arccot \lambda_n+ \arccot\lambda_{n-1}> 2 \arccot C_1^{-1}>\Theta_0-c_0$ if $C_1$ is chosen to be large. Hence, we may assume $\lambda_1>2C_1$ so that $1-\lambda_1 \lambda_i <0$ for $1<i<n$ since otherwise the final result holds trivially. Combining this \eqref{2nd-order-longtime eqn 1} and \eqref{2nd-order-longtime eqn 2}, we arrive at
\[
\begin{split}
\Box_L \log\lambda_1
&\leq \frac{1-\lambda_1 \lambda_n}{1+\lambda_1^2} F_n|\nabla_n \log \lambda_1|^2+ F_1 |\nabla_1 \log \lambda_1|^2+C_0.
\end{split}
\]
Since $G$ attains its maximum at $(x_0,t_0)$,
\[
\nabla \log \lambda_1=-\phi' \nabla u,
\]
so that
\begin{equation}\label{logG-equ-1}
\begin{split}
0\leq \Box_LG &\leq \frac{1-\lambda_1 \lambda_n}{1+\lambda_1^2} F_n(\phi')^2 |\nabla_n u|^2+ F_1(\phi')^2 |\nabla_1 u|^2 +C_0\\
&\quad +\phi' \Box_L u-\phi'' F_i |\nabla_i u|^2.
\end{split}
\end{equation}

Let $\rho,K$ be the constant obtained from Lemma~\ref{sub-sol-TdHYMFlow}. Using Lemma \ref{lemma:parabolic-of-L}, we have $\lambda_{n}\geq-\tan\theta_{0}$ and $Q>c_0$. Then
\[
c_0 \leq \sum_{i=1}^n \arccot(\lambda_i) \leq n \arccot(\lambda_n)
\]
and hence
\begin{equation}\label{bdd-un}
|\lambda_n|\leq C_2.
\end{equation}
Denote $\arccot\lambda_{i}$ by $\theta_{i}$. Then $F_{i}$ can be written as
\[
F_i=\frac{\csc^2 Q}{1+\lambda_i^2}-\frac{\ve \csc^2 Q}{1+\lambda_i^2}\cdot \frac{\sin (Q-\theta_i)}{\prod_{j\neq i}\sqrt{1+\lambda_j^2}}.
\]
We see that if $\ve_0$ is sufficiently small, then for all $i$,
\begin{equation}\label{sim-control-Fi}
\frac{\csc^2 Q}{1+\lambda_i^2} \geq F_i \geq \frac{1}{2}\cdot\frac{\csc^2 Q}{1+\lambda_i^2}.
\end{equation}
Therefore if we further assume
\[
\lambda_1 \geq \max\left\{K+\max_X|\a_{\hat\varphi_0}|_\omega, \frac{10\rho}{1+C_2^2} \right\}
\]
then $|\a_\varphi-\a_{\hat\varphi_0}|_\omega>K$ and $F_1 \leq \rho F_n \leq \rho\sum_{i=1}^n F_i$. Combining this with Lemma \ref{sub-sol-TdHYMFlow} and \eqref{sim-control-Fi},
\begin{equation}\label{lower bound of box u}
\Box_L u > \rho \sum_{i=1}^n F_i \geq \rho F_n \geq \frac{\rho}{2}\cdot\frac{\csc^2 Q}{1+\lambda_n^2}> C_3^{-1},
\end{equation}
where we used Lemma \ref{lemma:parabolic-of-L} and \eqref{bdd-un} in the last inequality.

We also choose $C_4>0$ so that $C_4>\sup_{X\times [0,+\infty)}|u|$ which is possible thanks to Lemma~\ref{0rd-order-longtime}. Pick $\phi(s)=\frac{s^2}{2}-(C_5+C_4)s$ with $C_5=\max\{ 1,C_4,C_3(1+C_0)\}$ so that $2C_5>-\phi'>C_5$ and $\phi''=1$ at $s=u$. Substituting it and \eqref{lower bound of box u} into \eqref{logG-equ-1}, we have
\[
\begin{split}
0&\leq \frac{1-\lambda_1 \lambda_n}{1+\lambda_1^2} F_n(\phi')^2 |\nabla_n u|^2+ F_1(\phi')^2 |\nabla_1 u|^2 +C_0+C_3^{-1}\phi'-\phi'' F_i |\nabla_i u|^2\\
&\leq F_n|\nabla_n u|^2 \left(\frac{1-\lambda_1\lambda_n}{1+\lambda_1^2}-1\right)+4C_5^2 F_1 |\nabla_1u|^2+(C_0-C_3^{-1}C_5)\\[1mm]
&\leq F_n|\nabla_n u|^2 \left(C_2\lambda_1^{-1}-1\right)+{C_6}{\lambda_1^{-2}}|\nabla_1u|^2-1,
\end{split}
\]
where we have used \eqref{bdd-un}, \eqref{sim-control-Fi} and Lemma~\ref{lemma:parabolic-of-L} in the last inequality. Hence, if $\lambda_1>2C_2$, the first bracket in evolution equation is negative and thus
\[
\begin{split}
\lambda_1(x_0,t_0)\leq C_7\sup_{X\times [0,T]}|\nabla u|.
\end{split}
\]

In conclusion, we have shown that
$$\lambda_1(x_0,t_0)\leq \max\left\{ C_7\sup_{X\times [0,T]}|\nabla u|,2C_2,K+\max_X|\a_{\hat\varphi_0}|_\omega, \frac{10\rho}{1+C_2^2},2C_1 \right\}.$$
Since $\nabla u$ is uniformly comparable to $\nabla \varphi$, the assertion now follows by applying maximum principle on $G$ since $u$ is uniformly bounded for all $t\geq 0$.
\end{proof}

With Lemma~\ref{2nd-order-longtime}, we can control the complex Hessian of $\varphi$ using exactly the same blowup argument as in \cite[Lemma 5.7]{Takahashi2020}, see also \cite[Proposition 5.1]{CollinsJacobYau2020}.
\begin{lma}\label{2nd-order-longtime-true}
There exist $C,\ve_0>0$ depending only on $\varphi_0,\hat\varphi_0,\a,\omega,X$ such that for all $\ve<\ve_0$,
$$\sup_{X\times [0,T]}|\ddb\varphi|\leq C.$$
\end{lma}
\begin{proof}
Since the argument in  \cite[Lemma 5.7]{Takahashi2020} is only based on zero order estimate, quasi-subharmonicity and the fact that the complex Hessian is controlled by its gradient quadratically. The proof can be carried over thanks to Lemma~\ref{lemma:parabolic-of-L}, \ref{0rd-order-longtime} and \ref{2nd-order-longtime},
\end{proof}

\begin{proof}[Proof of Theorem~\ref{smooth convergence}]
This is by now standard using the concavity. We include a sketch of the proof for the reader's convenience. Thanks to Lemma~\ref{lemma:parabolic-of-L}, \ref{0rd-order-longtime} and \ref{2nd-order-longtime}, $\varphi,\dot\varphi$ are uniformly bounded for all $t\geq 0$. By \cite{Evans1982,Krylov1982}, the concavity of $F_{\omega,\ve}$ implies a uniform $C^{2,\a}$ estimates for all $t\geq 0$, e.g., see the proof of \cite[Theorem 4.1]{Gill2011} which only relies on the concavity. The higher order estimates follows from standard parabolic theory.

It remains to establish the convergence. Since $\dot\varphi$ satisfies $$\Box_L \dot\varphi=0$$ and $F^{i\bar j}_{\omega,\ve}\partial_i\partial_{\bar j}$ is uniformly elliptic, we can apply the standard argument (see e.g. \cite[Section 6]{Gill2011}) to show
\[
\mathrm{osc}_X\dot\varphi\leq Ce^{-C^{-1}t}
\]
for some $C(\varphi_0,\hat\varphi_0,\a,\omega,X)>0$. Moreover, by \eqref{twisted dHYM flow 2} and \eqref{CY-const eqn},
\[
\int_X \dot{\vp} \im\left(\a_{\varphi}+\sqrt{-1}\omega \right)^n = 0.
\]
This shows for all $t\geq 0$, there is $x_t\in X$ so that $\dot\varphi(x_t,t)=0$. It then follows that $\dot\varphi$ decays to zero exponentially fast as $t\to +\infty$. Together with the boundedness of $\varphi$, it is clear that $\varphi(t)\to \varphi_{\infty,\ve}$ as $t\to+\infty$ for some function $\varphi_{\infty,\ve}$. By the higher order estimates, this convergence is $C^{\infty}$ and so $\varphi_{\infty,\ve}$ is a smooth function which solves the twisted dHYM equation
$$F_{\omega,\ve}(\a_{\varphi_{\infty,\ve}})=\cot\theta_0 +a_0\ve.$$
This completes the proof.
\end{proof}

\section{Technical lemmas}\label{Section: Technical lemmas}
In this section, we will prove some technical lemmas, which will be used in the proof of Theorem \ref{main theorem}.

\subsection{Continuity of operators}
In this subsection, we want to discuss the continuity of the operators $P_\omega,Q_\omega$ and $F_{\omega,\ve}$ with respect to $\omega$. We begin with the continuity of $P_\omega$ and $Q_\omega$.

\begin{lma}\label{Q P continuity}
For any $c_{0}\in(0,\pi)$, there exists a constant $\sigma_{0}(c_{0},n)$ such that the following holds. Let $\omega_1$ and $\omega_{2}$ be two metrics and $\alpha$ be a real $(1,1)$-form satisfying
\[
(1-\sigma^{5})\omega_{1} \leq \omega_{2} \leq (1+\sigma^{5})\omega_{1}
\]
for some $\sigma\in(0,\sigma_{0})$. Then
\begin{enumerate}\setlength{\itemsep}{1mm}
\item[(i)] $Q_{\omega_{2}}(\alpha) \leq Q_{\omega_{1}}(\alpha)+\sigma$ if $Q_{\omega_{1}}(\alpha)\in(0,\pi-c_{0})$;
\item[(ii)]$P_{\omega_{2}}(\alpha) \leq P_{\omega_{1}}(\alpha)+\sigma$ if $P_{\omega_{1}}(\alpha)\in(0,\pi-c_{0})$.
\end{enumerate}
\end{lma}

\begin{proof}
The proofs of (i) and (ii) are similar, which are almost identical to that of \cite[Proposition 2.5]{ChuLeeTakahashi2021}. For the reader's convenience, we give a sketch of the proof of (i). Since
\[
Q_{\omega_{2}}(\alpha)\leq Q_{\omega_{2}}(\alpha+\sigma\omega_{2})+n\sigma,
\]
after relabelling, it suffices to prove
\[
Q_{\omega_{2}}(\alpha+\sigma\omega_{2}) \leq Q_{\omega_{1}}(\alpha)+\sigma.
\]
Write $\theta=Q_{\omega_{1}}(\alpha)$ and $\theta'=\theta+\sigma$. The above inequality is equivalent to
\[
\Im\left( e^{-\sqrt{-1}\theta'}(\alpha+\sigma\omega_2+\sqrt{-1}\omega_2)^n \right) \leq 0.
\]

Set $\omega_{d}=\omega_{2}-\omega_{1}$ so that
\[
\alpha+\sigma\omega_2+\sqrt{-1}\omega_2
= (\alpha+\sqrt{-1}\omega_{1})+(\sigma\omega_{2}+\sqrt{-1}\omega_{d}).
\]
We compute
\[
\begin{split}
& \Im\left( e^{-\sqrt{-1}\theta'}(\alpha+\sigma\omega_2+\sqrt{-1}\omega_2)^n \right) \\
= {} & \sum_{k=0}^n \binom{n}{k}\Im\left( e^{-\sqrt{-1}\theta'} (\alpha+\sqrt{-1}\omega_1)^{n-k} \wedge (\sigma \omega_2+\sqrt{-1} \omega_d)^k  \right).
\end{split}
\]
For each $0\leq k\leq n$, we write
\[
\mathbf{T}_{k} = \Im\left( e^{-\sqrt{-1}\theta'}
(\alpha+\sqrt{-1}\omega_1)^{n-k} \wedge (\sigma \omega_2+\sqrt{-1} \omega_d)^k  \right).
\]
It is clear that $\mathbf{T}_{0}\leq 0$ since $Q_{\omega_{1}}(\alpha)\leq\theta'$. It suffices to show that $\mathbf{T}_{k}\leq0$ for any $1\leq k\leq n$. We compute
\[
\begin{split}
\mathbf{T}_{k}
= {} & \sum_{l=0}^{k-1}\binom{k}{l}\Im\left( e^{-\sqrt{-1}\theta'}(\alpha+\sqrt{-1}\omega_{1})^{n-k} \wedge (\sigma \omega_{2})^{l} \wedge (\sqrt{-1}\omega_{d})^{k-l}\right) \\
& +\Im \left( e^{-\sqrt{-1}\theta'} (\alpha+\sqrt{-1}\omega_{1})^{n-k} \right)\wedge (\sigma\omega_{2})^{k}.
\end{split}
\]
By the almost identical argument of \cite[Claim 2.7]{ChuLeeTakahashi2021} (replacing $\omega,\chi_{2},\chi_{3},\chi_{d}$ by $\alpha,\omega_{1},\omega_{2},\omega_{d}$),
\[
\begin{split}
& \Im\left( e^{-\sqrt{-1}\theta'}(\alpha+\sqrt{-1}\omega_{1})^{n-k} \wedge (\sigma \omega_{2})^{l} \wedge (\sqrt{-1}\omega_{d})^{k-l}\right) \\
\leq {} & -2\sigma^{k+2}\Im\left(e^{-\sqrt{-1}\theta'}(\alpha+\sqrt{-1}\omega_{1})^{n-k} \right) \wedge \omega_{2}^{k}.
\end{split}
\]
Then we obtain
\[
\mathbf{T}_{k} \leq (\sigma^{k}-C_{n}\sigma^{k+2})
\Im\left(e^{-\sqrt{-1}\theta'}(\alpha+\sqrt{-1}\omega_{1})^{n-k} \right) \wedge \omega_{2}^{k} \leq 0,
\]
for some constant $C_{n}$ depending only on $n$. Choosing $\sigma_{0}$ sufficiently small, we are done.
\end{proof}

\bigskip

\begin{lma}\label{F continuity}
Let $\omega_1$ and $\omega_{2}$ be two metrics, and $\alpha$ be a real $(1,1)$-form. Suppose that $Q_{\omega_{1}}(\alpha)\in(c_{0},\pi-c_{0})$. There exists a constant $\ve_{0}(c_{0},n)$ such that the following holds. For any $\ve\in(0,\ve_{0})$, there exists $\sigma_{\ve}(\ve,c_{0},n)$ such that if
\[
(1-\sigma^{5})\omega_{1} \leq \omega_{2} \leq (1+\sigma^{5})\omega_{1}
\]
for some $\sigma\in(0,\sigma_{\ve})$, then
\[
F_{\omega_{2},\ve}(\alpha) \geq F_{\omega_{1},\ve}(\alpha)-\ve^{2}.
\]
\end{lma}

\begin{proof}
For $k=1,2$, recall
\begin{equation}\label{F continuity eqn 4}
F_{\omega_{k},\ve}(\alpha)
= \cot(Q_{\omega_{k}}(\alpha))+\frac{\ve\omega_{k}^{n}}
{\Im(\alpha+\sqrt{-1}\omega_{k})^{n}}.
\end{equation}
For the first term in \eqref{F continuity eqn 4}, using Lemma \ref{Q P continuity} twice,
\begin{equation}\label{F continuity eqn 1}
|Q_{\omega_{1}}(\alpha)-Q_{\omega_{2}}(\alpha)| \leq \sigma, \quad
\frac{c_{0}}{2} < Q_{\omega_{1}}(\alpha),\; Q_{\omega_{2}}(\alpha) < \pi-\frac{c_{0}}{2},
\end{equation}
and so
\begin{equation}\label{F continuity eqn 2}
|\cot(Q_{\omega_{1}}(\alpha))-\cot(Q_{\omega_{2}}(\alpha))| \leq C\sigma.
\end{equation}

Next we deal with the second term of \eqref{F continuity eqn 4}. Let $\lambda_{1}\geq\ldots\geq\lambda_{n}$ and $\mu_{1}\geq\ldots\geq\mu_{n}$ be eigenvalues of $\alpha$ with respect to $\omega_{1}$ and $\omega_{2}$. Then
\begin{equation}\label{F continuity eqn 3}
\frac{\ve\omega_{1}^{n}}{\Im(\alpha+\sqrt{-1}\omega_{1})^{n}}
= \frac{\ve}{\sin(Q_{\omega_{1}}(\alpha))\prod_{i=1}^{n}\sqrt{1+\lambda_{i}^{2}}}
\end{equation}
and
\[
\frac{\ve\omega_{2}^{n}}{\Im(\alpha+\sqrt{-1}\omega_{2})^{n}}
= \frac{\ve}{\sin(Q_{\omega_{2}}(\alpha))\prod_{i=1}^{n}\sqrt{1+\mu_{i}^{2}}}.
\]
We split the proof into two cases.

\bigskip
\noindent
{\bf Case 1:} $\lambda_{1}\geq\ve^{-3}$.
\bigskip

By \eqref{F continuity eqn 1} and \eqref{F continuity eqn 3}, $\lambda_{1}\geq\ve^{-3}$ implies
\[
\frac{\ve\omega_{1}^{n}}{\Im(\alpha+\sqrt{-1}\omega_{1})^{n}}
\leq \ve^{3}.
\]
Combining this with \eqref{F continuity eqn 2},
\[
F_{\omega_{2},\ve}(\alpha)-F_{\omega_{1},\ve}(\alpha)
\geq -C\sigma-\frac{\ve\omega_{1}^{n}}{\Im(\alpha+\sqrt{-1}\omega_{1})^{n}}
\geq -C\sigma-\ve^{3}.
\]
After choosing $\ve_{0}$ and $\sigma_{\ve}$ sufficiently small, we are done.

\bigskip
\noindent
{\bf Case 2:} $\lambda_{1}<\ve^{-3}$.
\bigskip

Since $Q_{\omega_{1}}(\alpha)\in(c_{0},\pi-c_{0})$, we have $|\lambda_{i}|\leq C\lambda_{1}$ for each $i$. By Weyl's inequality, for each $i$
\[
|\lambda_{i}-\mu_{i}| \leq C\sigma\lambda_{1} \leq \ve^{-3}\sigma.
\]
By choosing $\sigma_{\ve}$ sufficiently small, for any $\sigma\in(0,\sigma_{\ve})$, we have $|\lambda_{i}-\mu_{i}|\leq\ve^{5}$. Combining this with \eqref{F continuity eqn 1},
\[
\left|\frac{\ve}{\sin(Q_{\omega_{2}}(\alpha))\prod_{i=1}^{n}\sqrt{1+\mu_{i}^{2}}}
-\frac{\ve}{\sin(Q_{\omega_{1}}(\alpha))\prod_{i=1}^{n}\sqrt{1+\lambda_{i}^{2}}}\right|
\leq C\ve\sigma+C\ve^{4}.
\]
Recalling \eqref{F continuity eqn 2}, the above shows
\[
F_{\omega_{2},\ve}(\alpha)-F_{\omega_{1},\ve}(\alpha) \geq -C\sigma-C\ve^{4}.
\]
After shrinking $\ve_{0}$ and $\sigma_{\ve}$ if necessary, we are done.
\end{proof}

\begin{lma}\label{Q F P relationship}
For $0<\theta<\Theta<\pi$, there exist constants $c_{0}(\theta,\Theta,n)$ and $\ve_{0}(\theta,\Theta,n)$ such that if $\ve\in(0,\ve_{0})$, $Q_{\omega}(\alpha)<\Theta$ and $F_{\omega,\ve}(\alpha)\geq\cot(\theta)+\ve$, then
\[
P_{\omega}(\alpha) \leq \theta-c_{0}\ve.
\]
\end{lma}

\begin{proof}
It suffices to show that
\[
\cot(P_{\omega}(\alpha)) \geq \cot(\theta)+\frac{\ve}{2}.
\]
Indeed, this implies
\[
P_{\omega}(\alpha) \leq \arccot\left(\cot(\theta)+\frac{\ve}{2}\right) \leq \theta-c_{0}\ve.
\]

Let $\lambda_{1}\geq\ldots\geq\lambda_{n}$ be the eigenvalues of $\alpha$ with respect to $\omega$. Then we have
\[
P_{\omega}(\alpha) = \sum_{i=2}^{n}\arccot(\lambda_{i}), \ \
Q_{\omega}(\alpha) = \arccot(\lambda_{1})+P_{\omega}(\alpha)
\]
and
\begin{equation}\label{Q F P relationship eqn 1}
F_{\omega,\ve}(\alpha) = \cot(Q_{\omega}(\alpha))+\frac{\ve}{\sin(Q_{\omega}(\alpha))
\prod_{i=1}^{n}\sqrt{1+\lambda_{i}^{2}}}.
\end{equation}
Without loss of generality, we may assume 
\begin{equation}\label{Q F P relationship eqn 4}
0 < \frac{\theta}{2} \leq P_{\omega}(\alpha) < Q_{\omega}(\alpha) < \Theta < \pi.
\end{equation}
This shows
\[
\sin(Q_{\omega}(\alpha)) \geq \frac{1}{C_{1}} > 0.
\]
We split the proof into two cases.

\bigskip
\noindent
{\bf Case 1:} $\lambda_{1}\geq 2C_{1}$.
\bigskip

In this case, we have
\[
\sin(Q_{\omega}(\alpha))\prod_{i=1}^{n}\sqrt{\lambda_{i}^{2}+1}
\geq \frac{1}{C_{1}}\lambda_{1} \geq 2.
\]
Substituting this into \eqref{Q F P relationship eqn 1} and using assumption $F_{\omega,\ve}(\alpha)\geq\cot(\theta)+\ve$,
\[
\cot(\theta)+\ve \leq F_{\omega,\ve}(\alpha)
\leq \cot(Q_{\omega}(\alpha))+\frac{\ve}{2}
\leq \cot(P_{\omega}(\alpha))+\frac{\ve}{2},
\]
which implies
\[
\cot(P_{\omega}(\alpha)) \geq \cot(\theta)+\frac{\ve}{2}.
\]
as required.

\bigskip
\noindent
{\bf Case 2:} $\lambda_{1}<2C_{1}$.
\bigskip

In this case, \eqref{Q F P relationship eqn 4} shows $\frac{\theta}{2}\leq P_{\omega}(\alpha)<\Theta$. By the mean value theorem, there is a constant $C_{2}$ such that
\begin{equation}\label{Q F P relationship eqn 2}
\begin{split}
\cot(Q_{\omega}(\alpha)) = {} & \cot(P_{\omega}(\alpha)+\arccot(\lambda_{1})) \\
\leq {} & \cot(P_{\omega}(\alpha))-C_{2}^{-1}\arccot(\lambda_{1}) \\
\leq {} & \cot(P_{\omega}(\alpha))-C_{2}^{-1}\arccot(2C_{1}).
\end{split}
\end{equation}
On the other hand, using \eqref{Q F P relationship eqn 4} again, we obtain $\frac{\theta}{2}\leq Q_{\omega}(\alpha)<\Theta$ and so
\begin{equation}\label{Q F P relationship eqn 3}
\frac{\ve}{\sin(Q_{\omega}(\alpha))
\prod_{i=1}^{n}\sqrt{1+\lambda_{i}^{2}}} \leq C_{3}\ve.
\end{equation}
Using \eqref{Q F P relationship eqn 2}, \eqref{Q F P relationship eqn 3} and assumption $F_{\omega,\ve}(\alpha)\geq\cot(\theta)+\ve$,
\[
\cot(\theta)+\ve \leq F_{\omega,\ve}(\alpha) \leq \cot(P_{\omega}(\alpha))-C_{2}^{-1}\arccot(2C_{1})+C_{3}\ve.
\]
It then follows that
\[
\cot(P_{\omega}(\alpha)) \geq \cot(\theta)+C_{2}^{-1}\arccot(2C_{1})-C_{3}\ve.
\]
Choosing $\ve_{0}$ sufficiently small, so that
\[
\cot(P_{\omega}(\alpha)) \geq \cot(\theta)+\frac{\ve}{2},
\]
for any $\ve\in(0,\ve_{0})$, as required.
\end{proof}

\subsection{Some inequalities}
\begin{lma}\label{s Im control}
There exists $C(\alpha,\omega,X)$ such that for any $\vp\in\H$ and $s\in[\frac{1}{2},1]$,
\[
\Im(\alpha_{s\vp}+\sqrt{-1}\omega)^{n} \geq \frac{1}{C}\Im(\alpha_{\vp}+\sqrt{-1}\omega)^{n}.
\]
\end{lma}

\begin{proof}
Let $\lambda_{1}\geq\ldots\geq\lambda_{n}$ and $\mu_{1}\geq\ldots\geq\mu_{n}$ be the eigenvalues of $\alpha_{\vp}$ and $\alpha_{s\vp}$ with respect to $\omega$. Let $A_{0}$ be the constant such that $-A_{0}\omega\leq \alpha\leq A_{0}\omega$. Using $\alpha_{s\vp}=(1-s)\alpha+s\alpha_{\vp}$ and Weyl's inequality, for each $i$,
\begin{equation}\label{s Im control eqn 1}
|\mu_{i}-s\lambda_{i}| \leq A_{0}.
\end{equation}
Since $s\in[\frac{1}{2},1]$, the above shows
\[
\frac{\Im(\alpha_{s\vp}+\sqrt{-1}\omega)^{n}}{\Im(\alpha_{\vp}+\sqrt{-1}\omega)^{n}}
= \frac{\sin(Q_{\omega}(\alpha_{s\vp}))}{\sin(Q_{\omega}(\alpha_{\vp}))}
\cdot\prod_{i=1}^{n}\sqrt{\frac{1+\mu_{i}^{2}}{1+\lambda_{i}^{2}}}
\geq \frac{1}{C}\cdot\frac{\sin(Q_{\omega}(\alpha_{s\vp}))}{\sin(Q_{\omega}(\alpha_{\vp}))}.
\]
We split the proof into two cases.

\bigskip
\noindent
{\bf Case 1:} $\lambda_{n}<100A_{0}$.
\bigskip

In this case, \eqref{s Im control eqn 1} shows $\mu_{n}<100A_{0}$ and
\[
\Theta_{0} > Q_{\omega}(\alpha_{s\vp}) > \arccot(\mu_{n}) \geq \arccot(100A_{0}) > 0.
\]
which implies
\[
\frac{\Im(\alpha_{s\vp}+\sqrt{-1}\omega)^{n}}{\Im(\alpha_{\vp}+\sqrt{-1}\omega)^{n}}
\geq \frac{1}{C}\cdot\frac{\sin(Q_{\omega}(\alpha_{s\vp}))}{\sin(Q_{\omega}(\alpha_{\vp}))}
\geq \frac{1}{C}.
\]

\bigskip
\noindent
{\bf Case 2:} $\lambda_{n}\geq100A_{0}$.
\bigskip

If $\lambda_{n}\geq100A_{0}$, then \eqref{s Im control eqn 1} shows that for each $i$,
\[
\frac{1}{4}\lambda_{i} \leq \mu_{i} \leq 2\lambda_{i}.
\]
Increasing $A_{0}$ if necessary, we have
\[
Q_{\omega}(\omega_{s\vp})
\leq \sum_{i=1}^{n}\arccot\left(\frac{\lambda_{i}}{4}\right)
\leq n\,\arccot\left(25A_{0}\right) \leq \frac{\pi}{2}
\]
and
\[
Q_{\omega}(\omega_{s\vp}) \geq \sum_{i=1}^{n}\arccot(2\lambda_{i})
\geq \frac{1}{C}\sum_{i=1}^{n}\arccot(\lambda_{i})
\geq \frac{1}{C}\cdot Q_{\omega}(\alpha_{\vp}) > 0.
\]
Then
\[
\sin(Q_{\omega}(\alpha_{s\vp}))
\geq \sin\left(\frac{1}{C}\cdot Q_{\omega}(\alpha_{\vp})\right)
\geq \frac{1}{C}\sin(Q_{\omega}(\alpha_{\vp})),
\]
which implies
\[
\frac{\Im(\alpha_{s\vp}+\sqrt{-1}\omega)^{n}}{\Im(\alpha_{\vp}+\sqrt{-1}\omega)^{n}}
\geq \frac{1}{C}\cdot\frac{\sin(Q_{\omega}(\alpha_{s\vp}))}{\sin(Q_{\omega}(\alpha_{\vp}))}
\geq \frac{1}{C}.
\]
\end{proof}

\begin{lma}\label{dp distance inequality 3}
For $c_{0}\in(0,\Theta_{0})$, there exists $C(c_{0},\alpha,\omega,X)$ such that
for any $\vp\in\H$ with $\sup_{X}\vp=0$ and $Q_{\omega}(\alpha_{\vp})\leq\Theta_{0}-c_{0}$,
\[
\int_{0}^{1}\int_{X}|\vp|^{p}\Im(\alpha_{t\vp}+\sqrt{-1}\omega)^{n}dt
\leq Cd_{p}^{p}(\vp,0).
\]
\end{lma}

\begin{proof}
By Lemma \ref{s Im control} (choosing $s=\frac{1}{2}$), for any $t\in[0,1]$, we have
\[
\Im(\alpha_{t\vp}+\sqrt{-1}\omega)^{n}
\leq C\Im(\alpha_{\frac{t\vp}{2}}+\sqrt{-1}\omega)^{n}.
\]
It then follows that
\begin{equation}\label{dp distance inequality 3 eqn 1}
\begin{split}
& \int_{0}^{1}\int_{X}|\vp|^{p}\Im(\alpha_{t\vp}+\sqrt{-1}\omega)^{n}dt \\
= {} & \int_{0}^{\frac{1}{2}}\int_{X}|\vp|^{p}\Im(\alpha_{t\vp}+\sqrt{-1}\omega)^{n}dt
+\int_{\frac{1}{2}}^{1}\int_{X}|\vp|^{p}\Im(\alpha_{t\vp}+\sqrt{-1}\omega)^{n}dt \\
\leq {} & \int_{0}^{\frac{1}{2}}\int_{X}|\vp|^{p}\Im(\alpha_{t\vp}+\sqrt{-1}\omega)^{n}dt
+C\int_{\frac{1}{2}}^{1}\int_{X}|\vp|^{p}\Im(\alpha_{\frac{t\vp}{2}}+\sqrt{-1}\omega)^{n}dt \\
\leq {} & C\int_{0}^{\frac{1}{2}}\int_{X}|\vp|^{p}\Im(\alpha_{t\vp}+\sqrt{-1}\omega)^{n}dt.
\end{split}
\end{equation}
On the other hand, by Lemma \ref{dp distance inequality 1} and \ref{H dp lemma} (iii),
\begin{equation}\label{dp distance inequality 3 eqn 2}
\begin{split}
& d_{p}^{p}(\vp,0) = 2\int_{0}^{\frac{1}{2}}d_{p}^{p}(\vp,0)dt
\geq 2\int_{0}^{\frac{1}{2}}d_{p}^{p}(\vp,t\vp)dt \\
\geq {} & 2\int_{0}^{\frac{1}{2}}\int_{X}|\vp-t\vp|^{p}
\Re\left(e^{-\sqrt{-1}\theta_{0}}(\alpha_{t\vp}+\sqrt{-1}\omega)^{n}\right)dt \\
\geq {} & \frac{1}{2^{p-1}}\int_{0}^{\frac{1}{2}}\int_{X}|\vp|^{p}
\Re\left(e^{-\sqrt{-1}\theta_{0}}(\alpha_{t\vp}+\sqrt{-1}\omega)^{n}\right)dt.
\end{split}
\end{equation}
For any $t\in[0,\frac{1}{2}]$, we have $\alpha_{t\vp} = (1-t)\alpha+t\alpha_{\vp}$. Suppose that $Q_{\omega}(\alpha)\leq\Theta_{0}-c_{0}'$. Thanks to the concavity of $\cot(Q_{\omega})$ and $Q_{\omega}(\alpha_{\vp})\leq\Theta_{0}-c_{0}$, we obtain
\[
Q_{\omega}(\alpha_{t\vp}) \leq \Theta_{0}-c_{0}'' = \theta_{0}+\frac{\pi}{2}-c_{0}'',
\]
which implies
\[
\cos(Q_{\omega}(\alpha_{t\vp})-\theta_{0}) \geq \frac{1}{C}
\geq \frac{1}{C}\sin(Q_{\omega}(\alpha_{t\vp})).
\]
It then follows that
\[
\Re\left(e^{-\sqrt{-1}\theta_{0}}(\alpha_{t\vp}+\sqrt{-1}\omega)^{n}\right)
\geq \frac{1}{C}\Im(\alpha_{t\vp}+\sqrt{-1}\omega)^{n}.
\]
Combining this with \eqref{dp distance inequality 3 eqn 1} and \eqref{dp distance inequality 3 eqn 2},
\[
d_{p}^{p}(\vp,0)
\geq \frac{1}{C}\int_{0}^{\frac{1}{2}}\int_{X}|\vp|^{p}\Im(\alpha_{t\vp}+\sqrt{-1}\omega)^{n}dt
\geq \frac{1}{C}\int_{0}^{1}\int_{X}|\vp|^{p}\Im(\alpha_{t\vp}+\sqrt{-1}\omega)^{n}dt,
\]
as required.
\end{proof}

\section{Proof of $(1)\Rightarrow(2)\Rightarrow(3)\Rightarrow(4)$ in Theorem \ref{main theorem}} \label{Section: main result 1}
In this section, we give the proof of $(1)\Rightarrow(2)\Rightarrow(3)\Rightarrow(4)$ in Theorem \ref{main theorem}.

\begin{proof}[Proof of $(1)\Rightarrow(2)\Rightarrow(3)$ in Theorem \ref{main theorem}]
The part $(2)\Rightarrow(3)$ is trivial. For $(1)\Rightarrow(2)$, it suffices to prove the case when $c$ is small. Without loss of generality, we assume that $c<\pi-\Theta_{0}$. For any $\vp\in\H_{c}$ with $\sup_{X}\vp=0$, we have $Q_{\omega}(\alpha_{\vp})\in(c,\Theta_{0})$ and so
\[
\sin(Q_{\omega}(\alpha_{\vp})) > \sin(c).
\]
Let $\lambda_{i}$ be the eigenvalues of $\alpha_{\vp}$ with respect to $\omega$. Then
\[
\begin{split}
& \frac{\Re\left(e^{-\sqrt{-1}\theta_{0}}(\alpha_{\vp}+\sqrt{-1}\omega)^{n}\right)}{\omega^{n}}
= \cos(Q_{\omega}(\alpha_{\vp})-\theta_{0})\prod_{i=1}^{n}\sqrt{1+\lambda_{i}^{2}} \\
< {} & \frac{\sin(Q_{\omega}(\alpha_{\vp}))}{\sin(c)}\cdot\prod_{i=1}^{n}\sqrt{1+\lambda_{i}^{2}}
= \frac{1}{\sin(c)}\cdot\frac{\Im(\alpha_{\vp}+\sqrt{-1}\omega)^{n}}{\omega^{n}}.
\end{split}
\]
Combining this with Lemma \ref{dp distance inequality 1},
\[
d_{1}(\vp,0) \leq \frac{1}{\sin(c)}\int_{X}|\vp|\Im(\alpha_{\vp}+\sqrt{-1}\omega)^{n}.
\]
Using $\vp\in\PSH(X,\chi)$ from Lemma \ref{quasi-plurisubharmonic} and $\sup_{X}\vp=0$,
\[
\int_{X}|\vp|\Im(\alpha+\sqrt{-1}\omega)^{n} \leq C.
\]
Then the coerciveness of $\J$ shows
\[
\J(\vp) \geq  \delta\int_{X}|\vp|\Im(\alpha_{\vp}+\sqrt{-1}\omega)^{n}-C
\geq \delta\sin(c)d_{1}(\vp,0)-C,
\]
which implies $\J$ is proper.
\end{proof}

We will prove $(3)\Rightarrow(4)$ in the rest of this section. For the dHYM equation \eqref{intro-dHYM}, by \cite[Theorem 1.2]{CollinsJacobYau2020}, the existence of subsolution implies the existence of solution. Therefore, it suffices to construct a subsolution.
\begin{thm}\label{subsolution}
If the $\J$-functional is weakly proper, then there exists $\ti{\vp}\in C^{\infty}(X)$ such that
$\alpha_{\ti{\vp}}\in\Gamma_{\omega,\theta_{0},\Theta_{0}}$ on $X$ for some $\Theta_0\in (\theta_0,\pi)$.
\end{thm}
We consider the twisted dHYM flow starting from zero function:
\begin{equation}\label{twisted dHYM flow 3}
\begin{cases}
\ \partial_t \varphi = F_{\omega,\ve}(\alpha_{\vp})-\cot(\theta_{0})-a_{0}\ve, \\[1mm]
\ \vp(0) = 0,
\end{cases}
\end{equation}
where the constant $\ve$ will be determined later.

\subsection{Lower bound of $\J_{\ve}$}
\begin{lma}\label{J ve functional lower bound}
There exist $\ve_{0}(\alpha,\omega,X)$ and $C(\alpha,\omega,X)$ such that the following holds. For any $\ve\in(0,\ve_{0})$, along the twisted dHYM flow \eqref{twisted dHYM flow 3}, we have
\[
\J_{\ve}(\vp) \geq \frac{1}{C}\int_{X}(-\vp)(\chi_{\vp}^{n}-\chi^{n})-C.
\]
In particular, $\J_{\ve}(\vp)\geq-C$.
\end{lma}

\begin{proof}
Write $u=\vp-\sup_{X}\vp$, then the required inequality is equivalent to
\[
\J_{\ve}(u) \geq \frac{1}{C}\int_{X}(-u)(\chi_{u}^{n}-\chi^{n})-C.
\]
Thanks to Lemma \ref{lemma:parabolic-of-L}, we have $u\in\PSH(X,\chi)$ and so $\int_{X}|u|\chi^{n}\leq C$. It then suffices to prove
\[
\J_{\ve}(u) \geq \frac{1}{C}\int_{X}|u|\chi_{u}^{n}-C.
\]
By the definition of $\J_{\ve}$ and $\J=\sin(\theta_{0})\J_{0}$ (see Lemma \ref{J J 0 relationship}), we have
\begin{equation}\label{J ve functional lower bound eqn 1}
\begin{split}
\J_{\ve}(u) = {} & \J_{0}(u)
+\ve\int_{0}^{1}\int_{X}u\left(a_{0}\Im(\alpha_{tu}+\sqrt{-1}\omega)^{n}-\omega^{n}\right)dt \\
= {} & \frac{\J(u)}{\sin(\theta_{0})}
-a_{0}\ve\int_{0}^{1}\int_{X}|u|\Im(\alpha_{tu}+\sqrt{-1}\omega)^{n}dt
+\ve\int_{X}|u|\omega^{n}.
\end{split}
\end{equation}
Lemma~\ref{lemma:parabolic-of-L} implies $Q_{\omega}(\alpha_{u})>c_{0}$. Then by the weakly properness of $\J$,
\[
\J(u) \geq \delta_{c_{0}}d_{1}(u,0)-A_{c_{0}}.
\]
Thanks to Lemma \ref{dp distance inequality 2} and \ref{dp distance inequality 3},
\[
\begin{split}
\J(u) \geq {} & \frac{\delta_{c_{0}}}{2}d_{1}(u,0)
+\frac{\delta_{c_{0}}}{2}d_{1}(u,0)-A_{c_{0}} \\
\geq {} & \frac{\delta_{c_{0}}}{2C}\int_{X}|u|\chi_{u}^{n}
+\frac{\delta_{c_{0}}}{2C}\int_{0}^{1}\int_{X}|u|\Im(\alpha_{tu}+\sqrt{-1}\omega)^{n}dt-A_{c_{0}}
\end{split}
\]
Substituting this into \eqref{J ve functional lower bound eqn 1},
\[
\begin{split}
\J_{\ve}(u) \geq {} &
\left(\frac{\delta_{c_{0}}}{2C\sin(\theta_{0})}-a_{0}\ve\right)
\int_{0}^{1}\int_{X}|u|\Im(\alpha_{tu}+\sqrt{-1}\omega)^{n}dt \\
& +\frac{\delta_{c_{0}}}{2C\sin(\theta_{0})}\int_{X}|u|\chi_{u}^{n}
-\frac{A_{c_{0}}}{\sin(\theta_{0})}.
\end{split}
\]
This completes the proof by choosing $\ve_{0}=\frac{\delta_{c_{0}}}{4a_{0}C\sin(\theta_{0})}$.
\end{proof}

\subsection{Limiting function $u_{\infty}$}
Thanks to Lemma \ref{J ve functional lower bound}, $\J_{\ve}(\vp)$ is bounded from below. Combining this with $\de_{t}\J_{\ve}(\vp)<0$, there exists a sequence $\{t_{k}\}_{k=1}^{\infty}$ such that $\lim_{k\rightarrow\infty}t_{k}=\infty$ and
\[
\lim_{k\rightarrow\infty}\frac{\de\J_{\ve}(\vp)}{\de t}\bigg|_{t=t_{k}} = 0.
\]
Write $\vp_{t_{k}}=\vp(\cdot,t_{k})$, then the above implies
\[
\lim_{k\rightarrow\infty}\int_{X}\left(F_{\omega,\ve}(\alpha_{\vp_{t_{k}}})-\cot(\theta_{0})-a_{0}\ve\right)^{2}
\Im(\alpha_{\vp_{t_{k}}}+\sqrt{-1}\omega)^{n} = 0.
\]
By \eqref{H c positivity} and Lemma~\ref{lemma:parabolic-of-L}, we obtain
\[
\Im(\alpha_{\vp_{t_{k}}}+\sqrt{-1}\omega)^{n} \geq \sin(c_{0})\omega^{n}.
\]
It then follows that
\[
\lim_{k\rightarrow\infty}\int_{X}\left(F_{\omega,\ve}(\alpha_{\vp_{t_{k}}})-\cot(\theta_{0})-a_{0}\ve\right)^{2}
\omega^{n} = 0.
\]
We normalize
\[
u_{k} = \vp_{t_{k}}-\sup_{X}\vp_{t_{k}},
\]
so that
\begin{equation}\label{L 2 convergence}
\lim_{k\rightarrow\infty}\int_{X}\left(F_{\omega,\ve}(\alpha_{u_{k}})-\cot(\theta_{0})-a_{0}\ve\right)^{2}
\omega^{n} = 0.
\end{equation}

By Lemma \ref{lemma:parabolic-of-L}, we have $\vp_{t_{k}}\in\PSH(X,\chi)$ and $u_{k}\in\PSH(X,\chi)$. Passing to a subsequence, we may assume that $u_{k}$ converges to $u_{\infty}$ in the $L^{1}$ sense, where $u_{\infty}\in\PSH(X,\chi)$. By Lemma \ref{J ve functional lower bound},
\[
\int_{X}(-u_{k})(\chi_{u_{k}}^{n}-\chi^{n})
= \int_{X}(-\vp_{t_{k}})(\chi_{\vp_{t_{k}}}^{n}-\chi^{n})
\leq C\J_{\ve}(\vp_{t_{k}})+C
\leq C\J_{\ve}(0)+C,
\]
where we have used that monotonicity of $\J_{\ve}$ along the twisted dHYM flow \eqref{twisted dHYM flow 3}.
Combining this with \cite[Proposition 2.2, Corollary 2.7]{GuedjZeriahi2007}, we obtain
\[
u_{\infty}\in\mathcal{E}^{1}(X,\chi)\subset\mathcal{E}(X,\chi).
\]
By \cite[Corollary 1.8]{GuedjZeriahi2007}, the function $u_{\infty}$ has zero Lelong number everywhere.

\subsection{Local regularization}

For $\sigma>0$, we choose a finite open cover $\{B_{i,4R}=B_{4R}(p_{i})\}$ of $X$, where $p_{i}\in X$ and $B_{4R}(p_{i})$ denotes the Euclidean ball with center $p_{i}$ and radius $4R$. We choose $R$ sufficiently small so that $\{B_{R}(p_i)\}$ is also a covering of $X$. On each $B_{i,4R}$, there exists a constant K\"ahler form $\omega_{i}$ such that
\[
(1-\sigma^{5})\omega_{i} \leq \omega \leq (1+\sigma^{5})\omega_{i}.
\]
We further assume that
\[
\alpha = \ddb\vp_{\alpha,i}, \ \
\omega = \ddb\vp_{\omega,i}, \ \
\chi = \ddb\vp_{\chi,i},
\]
where three potential functions $\vp_{\alpha,i}$, $\vp_{\omega,i}$ and $\vp_{\chi,i}$ satisfy
\[
(1-\sigma)|z|^{2} \leq \vp_{\omega,i} \leq (1+\sigma)|z|^{2}
\]
and
\[
|\vp_{\alpha,i}|+|\vp_{\omega,i}|+|\vp_{\chi,i}|
+|\nabla\vp_{\alpha,i}|+|\nabla\vp_{\omega,i}|+|\nabla\vp_{\chi,i}| \leq K
\]
for some constant $K>0$.

For any $k$, define
\[
u_{k,i} = u_{k}+\vp_{\alpha,i}, \ \
u_{\infty,i} = u_{\infty}+\vp_{\alpha,i}.
\]
It is clear that
\[
\ddb u_{k,i} = \ddb u_{k}+\alpha = \alpha_{u_{k}}.
\]

\begin{defn}
For any real $(1,1)$-form $\beta$ on $B_{i,R}$, the local regularization of $\beta$ is defined by
\[
\beta^{(r)}(x)=\int_{B_1(0)} r^{-2n}\rho\left(\frac{|y|}{r} \right)\beta(x-y) \; d\mathrm{vol}_{g_{\mathrm{E}}}(y)
\]
for $r\in (0,1)$, where $B_{1}(0)$ denotes the unit ball in $\mathbb{C}^{n}$, $g_{\mathrm{E}}$ denotes the Euclidean metric, and $\rho$ is a non-negative mollifier supporting on $[0,1]$ such that $\int_{B_{1}(0)}\rho \; d\mathrm{vol}_{g_{\mathrm{E}}}=1$.
\end{defn}

\begin{lma}\label{Q P local regularization}
There exist constants $\ve_{0}(\alpha,\omega,X)$ and $c_{0}(\alpha,\omega,X)$ such that the following holds. For $\ve\in(0,\ve_{0})$, there exist constant $\sigma_{0}(\ve,\alpha,\omega,X)$, $r_{0}(\ve,\alpha,\omega,X)$  such that for any $i$, $\ve\in(0,\ve_{0})$, $\sigma\in(0,\sigma_{0})$ and $r\in(0,r_{0})$,
\begin{enumerate}\setlength{\itemsep}{1mm}
\item[(i)] $Q_{\omega}(\ddb u_{\infty,i}^{(r)})\leq\Theta_{0}-\frac{c_{0}}{2}$;
\item[(ii)] $P_{\omega}(\ddb u_{\infty,i}^{(r)})\leq\theta_{0}-\frac{a_{0}c_{0}\ve}{2}$.
\end{enumerate}
\end{lma}

\begin{proof}
For (i), by definitions of $u_{k}$ and $u_{k,i}$, and Lemma~\ref{lemma:parabolic-of-L},
\begin{equation}\label{Q P local regularization eqn 1}
Q_{\omega}(\ddb u_{k,i}) = Q_{\omega}(\alpha_{u_{k}})
= Q_{\omega}(\alpha_{\vp_{t_{k}}}) \in (c_{0},\Theta_{0}-c_{0}).
\end{equation}
Using Lemma \ref{Q P continuity} (i),
\[
Q_{\omega_{i}}(\ddb u_{k,i}) \leq \Theta_{0}-c_{0}+\sigma \leq \Theta_{0}-\frac{2c_{0}}{3}.
\]
Thanks to the concavity of $\cot(Q_{\omega_{i}})$,
\begin{equation}\label{Q P local regularization eqn 2}
Q_{\omega_{i}}(\ddb u_{k,i}^{(r)}) = Q_{\omega_i}\big((\ddb u_{k,i})^{(r)}\big)
\leq \Theta_{0}-\frac{2c_{0}}{3}.
\end{equation}
Using Lemma \ref{Q P continuity} (i) again,
\[
Q_{\omega}(\ddb u_{k,i}^{(r)}) \leq \Theta_{0}-\frac{2c_{0}}{3}+\sigma
\leq \Theta_{0}-\frac{c_{0}}{2}.
\]
Letting $k\rightarrow\infty$, we obtain
\[
Q_{\omega}(\ddb u_{\infty,i}^{(r)}) < \Theta_{0}-\frac{c_{0}}{2}.
\]

\bigskip

For (ii), we first claim
\begin{equation}\label{Q P local regularization claim}
F_{\omega,\ve}(\ddb u_{\infty,i}^{(r)})
\geq \cot(\theta_{0})+\frac{2a_{0}\ve}{3}.
\end{equation}
To prove \eqref{Q P local regularization claim}, combining \eqref{Q P local regularization eqn 1} and Lemma \ref{F continuity},
\[
F_{\omega_{i},\ve}(\ddb u_{k,i}) \geq F_{\omega,\ve}(\ddb u_{k,i})-\ve^{2}
= F_{\omega,\ve}(\alpha_{u_{k}})-\ve^{2}.
\]
The concavity of $F_{\omega_{i},\ve}$ shows
\[
\begin{split}
& F_{\omega_{i},\ve}(\ddb u_{k,i}^{(r)})
= F_{\omega_{i},\ve}\big((\ddb u_{k,i})^{(r)}\big) \\[1mm]
\geq {} & \int_{B_1} r^{-2n}\rho\left(\frac{|y|}{r} \right) F_{\omega_{i},\ve}(\ddb u_{k,i})(x-y) d\mathrm{vol}_{g_{\mathrm{E}}}(y)\\
\geq {} & \int_{B_1} r^{-2n}\rho\left(\frac{|y|}{r} \right) F_{\omega,\ve}(\alpha_{u_{k}})(x-y) d\mathrm{vol}_{g_{\mathrm{E}}}(y)-\ve^{2}.
\end{split}
\]
Recall that $F_{\omega,\ve}(\alpha_{u_{k}})$ converges to $\cot(\theta_{0})+a_{0}\ve$ in $L^{2}$ (see \eqref{L 2 convergence}). Letting $k\rightarrow\infty$,
\[
F_{\omega_{i},\ve}(\ddb u_{\infty,i}^{(r)})
\geq \cot(\theta_{0})+a_{0}\ve-\ve^{2}.
\]
Decreasing $\ve_{0}$ if necessary, we obtain \eqref{Q P local regularization claim}.

By \eqref{Q P local regularization eqn 2} (letting $k\rightarrow\infty$), we obtain
\[
Q_{\omega_{i}}(\ddb u_{\infty,i}^{(r)}) < \Theta_{0}.
\]
Combining this with \eqref{Q P local regularization claim} and Lemma \ref{Q F P relationship},
\[
P_{\omega_{i}}(\ddb u_{\infty,i}^{(r)})\leq\theta_{0}-\frac{2a_{0}c_{0}\ve}{3}.
\]
Using Lemma \ref{Q P continuity} (ii), we obtain (ii).
\end{proof}

\subsection{Gluing argument}
In each $B_{i,3R}$, define the function $\ti{u}_{\infty,i}$ by
\[
\ti{u}_{\infty,i} = u_{\infty}+\vp_{\chi,i}.
\]
Since $u_{\infty}\in\PSH(X,\chi)$, then $\ti{u}_{\infty,i}$ is a plurisubharmonic function on $B_{i,3R}$. For $p\in B_{i,3R}$ and $r\in(0,R)$, we use $B_{i,r}(p)\subset B_{i,3R}$ to denote the Euclidean ball with center $p$ and radius $r$. Write
\[
M_{\ti{u}_{\infty,i}}(p,r) = \sup_{B_{i,r}(p)}\ti{u}_{\infty,i}, \ \
\nu_{\ti{u}_{\infty,i}}(p,r) = \frac{M_{\ti{u}_{\infty,i}}(p,R)-M_{\ti{u}_{\infty,i}}(p,r)}{\log R-\log r}.
\]
It is well-known that $\nu_{\ti{u}_{\infty,i}}(p,r)$ converges decreasingly to the Lelong number $\nu_{\ti{u}_{\infty,i}}(p)$ as $r\rightarrow0$, i.e.,
\[
\lim_{r\rightarrow0}\nu_{\ti{u}_{\infty,i}}(p,r) = \nu_{\ti{u}_{\infty,i}}(p).
\]
In fact, $\nu_{\ti{u}_{\infty,i}}(p)$ is independent of $i$. Precisely, for any $i$ such that $p\in B_{i,3R}$,
\[
\nu_{\ti{u}_{\infty,i}}(p) = \nu_{u_{\infty}}(p),
\]
where $\nu_{u_{\infty}}(p)$ denotes the Lelong number of $\chi$-plurisubharmonic function $u_{\infty}$ at $p$.

By adapting the idea from B{\l}ocki-Ko{\l}odziej \cite{BlockiKolodziej2007}, Chen \cite{Chen2021} proved the following lemma, which characterizes the behaviours of $M_{\ti{u}_{\infty,i}}(p,r)$ and $\ti{u}_{\infty,i}^{(r)}(p)$ with respect to $\nu_{\ti{u}_{\infty,i}}(p,r)$.

\begin{lma}[Lemma 4.2 of \cite{Chen2021}]\label{Chen lemma 4.2}
For any $p\in B_{i,3R}$ and $r\in(0,R)$, we have
\begin{enumerate}\setlength{\itemsep}{1mm}
\item[(i)] $0\leq M_{\ti{u}_{\infty,i}}(p,r)-M_{\ti{u}_{\infty,i}}(p,\frac{r}{2})\leq(\log2)\nu_{\ti{u}_{\infty,i}}(p,r)$,
\item[(ii)] $0\leq M_{\ti{u}_{\infty,i}}(p,r)-\ti{u}_{\infty,i}^{(r)}(p)\leq \tau_{0}\nu_{\ti{u}_{\infty,i}}(p,r)$, where $\tau_{0}$ is a universal constant depending only $n$.
\end{enumerate}
\end{lma}

By Lemma \ref{Q P local regularization}, we have
\[
Q_{\omega}(\ddb u_{\infty,i}^{(r)}-\ve^{2}\omega)
\leq Q_{\omega}(\ddb u_{\infty,i}^{(r)})+n\ve^{2}
\leq \Theta_{0}-\frac{a_{0}\ve}{2}+n\ve^{2}
\]
and
\[
P_{\omega}(\ddb u_{\infty,i}^{(r)}-\ve^{2}\omega)
\leq P_{\omega}(\ddb u_{\infty,i}^{(r)})+(n-1)\ve^{2}
\leq \theta_{0}-\frac{a_{0}c_{0}\ve}{2}+(n-1)\ve^{2}.
\]
Choosing $\ve$ sufficiently small (this fixes the value of $\ve$), we obtain
\[
\ddb u_{\infty,i}^{(r)}-\ve^{2}\omega \in \Gamma_{\omega,\theta_{0},\Theta_{0}}.
\]
Define
\[
\ti{\vp}_{i,r} = u_{\infty,i}^{(r)}-\vp_{\alpha,i}-\ve^{2}\vp_{\omega,i},
\]
so that
\begin{equation}\label{gluing eqn 1}
\alpha_{\ti{\vp}_{i,r}} \in \Gamma_{\omega,\theta_{0},\Theta_{0}}.
\end{equation}
Recalling the definitions of $u_{\infty,i}$ and $\ti{u}_{\infty,i}$, we have
\[
u_{\infty,i}^{(r)} = (u_{\infty}+\vp_{\alpha,i})^{(r)}
= (\ti{u}_{\infty,i}-\vp_{\chi,i}+\vp_{\alpha,i})^{(r)},
\]
which implies
\[
\ti{\vp}_{i,r}
= \ti{u}_{\infty,i}^{(r)}-\vp_{\chi,i}^{(r)}
+\vp_{\alpha,i}^{(r)}-\vp_{\alpha,i}-\ve^{2}\vp_{\omega,i}.
\]

\begin{lma}\label{gluing lemma}
There exists $r_{0}(\ve,\alpha,\omega,X)$ such that for any $r\in(0,r_{0})$ and any $p\in X$, we have
\[
\max_{\{j|p\in B_{j,3R}\setminus B_{j,2R}\}}\ti{\vp}_{j,r}(p)
< \max_{\{i|p\in B_{i,R}\}}\ti{\vp}_{i,r}(p)-\frac{\ve^{2}R^{2}}{2}.
\]
\end{lma}

\begin{proof}
Suppose that $p\in(B_{j,3R}\setminus B_{j,2R})\cap B_{i,R}$. For plurisubharmonic function $\ti{u}_{\infty,i}$ on $B_{i,3R}$, the maximum function $M_{\ti{u}_{\infty,i}}(\cdot,r)$ is continuous, and so $\nu_{\ti{u}_{\infty,i}}(\cdot,r)$ is also continuous. Since $\nu_{\ti{u}_{\infty,i}}(\cdot,r)$ is decreasing with respect to $r$ and the limit function is the Lelong number $\nu_{u_{\infty}}(\cdot)$. By Dini-Cartan lemma (see e.g. \cite[Lemma 2.2.9]{Helmsbook}), choosing $r_{0}$ sufficiently small, for any $r\in(0,r_{0})$,
\begin{equation}\label{nu i bound}
\nu_{\ti{u}_{\infty,i}}(p,r)
\leq \nu_{u_{\infty}}(p)+A^{-1}\ve^{2} = A^{-1}\ve^{2},
\end{equation}
where we used that $u_{\infty}$ has zero Lelong number everywhere and $A$ is a large constant to be determined later. Similarly,
\begin{equation}\label{nu j bound}
\nu_{\ti{u}_{\infty,j}}(p,r) \leq A^{-1}\ve^{2}.
\end{equation}

First, we establish the upper bound of $\ti{\vp}_{j,r}(p)$. Combining \eqref{nu j bound} and Lemma \ref{Chen lemma 4.2},
\[
\begin{split}
\ti{u}_{\infty,j}^{(r)}(p) \leq {} & M_{\ti{u}_{\infty,j}}(p,r)
\leq M_{\ti{u}_{\infty,j}}\left(p,\frac{r}{2}\right)+(\log 2)A^{-1}\ve^{2} \\[2mm]
= {} & \sup_{B_{j,\frac{r}{2}}(p)}\ti{u}_{\infty,j}+(\log 2)A^{-1}\ve^{2} \\
= {} & \sup_{B_{j,\frac{r}{2}}(p)}(u_{\infty}+\vp_{\chi,j})+(\log 2)A^{-1}\ve^{2} \\
\leq {} & \sup_{B_{j,\frac{r}{2}}(p)}u_{\infty}+\sup_{B_{j,\frac{r}{2}}(p)}\vp_{\chi,j}+A^{-1}\ve^{2}.
\end{split}
\]
Then we compute
\[
\begin{split}
\ti{\vp}_{j,r}(p)
= {} & \ti{u}_{\infty,j}^{(r)}(p)-\vp_{\chi,j}^{(r)}(p)
+\vp_{\alpha,j}^{(r)}(p)-\vp_{\alpha,j}(p)
-\ve^{2}\vp_{\omega,j}(p) \\
\leq {} & \sup_{B_{j,\frac{r}{2}}(p)}u_{\infty}+\left(\sup_{B_{j,\frac{r}{2}}(p)}\vp_{\chi,j}
-\vp_{\chi,j}(p)\right)+\left(\vp_{\chi,j}(p)-\vp_{\chi,j}^{(r)}(p)\right) \\
& +\left(\vp_{\alpha,j}^{(r)}(p)-\vp_{\alpha,j}(p)\right)+A^{-1}\ve^{2}-\ve^{2}\vp_{\omega,j}(p) \\
\leq {} & \sup_{B_{j,\frac{r}{2}}(p)}u_{\infty}+3Kr+A^{-1}\ve^{2}-3\ve^{2}R^{2}.
\end{split}
\]

Next, we apply the similar argument to establish the lower bound of $\vp_{i,r}(p)$. Combining \eqref{nu i bound} and Lemma \ref{Chen lemma 4.2},
\[
\begin{split}
\ti{u}_{\infty,i}^{(r)}(p) \geq {} & M_{\ti{u}_{\infty,i}}(p,r)-A^{-1}\tau_{0}\ve^{2} \\[2mm]
= {} & \sup_{B_{i,r}(p)}\ti{u}_{\infty,i}-A^{-1}\tau_{0}\ve^{2} \\
= {} & \sup_{B_{i,r}(p)}(u_{\infty}+\vp_{\chi,i})-A^{-1}\tau_{0}\ve^{2} \\
\geq {} & \sup_{B_{i,r}(p)}u_{\infty}+\inf_{B_{i,r}(p)}\vp_{\chi,i}-A^{-1}\tau_{0}\ve^{2}.
\end{split}
\]
Then we compute
\[
\begin{split}
\ti{\vp}_{i,r}(p) = {} & \ti{u}_{\infty,i}^{(r)}(p)-\vp_{\chi,i}^{(r)}(p)+\vp_{\alpha,i}^{(r)}(p)
-\vp_{\alpha,i}(p)-\ve^{2}\vp_{\omega,i}(p) \\
\geq {} & \sup_{B_{i,r}(p)}u_{\infty}
+\left(\inf_{B_{i,r}(p)}\vp_{\chi,i}-\vp_{\chi,i}(p)\right)
+\left(\vp_{\chi,i}(p)-\vp_{\chi,i}^{(r)}(p)\right) \\
& +\left(\vp_{\alpha,i}^{(r)}(p)-\vp_{\alpha,i}(p)\right)
-A^{-1}\tau_{0}\ve^{2}-\ve^{2}\vp_{\omega,i}(p) \\[1mm]
\geq {} & \sup_{B_{i,r}(p)}u_{\infty}-3Kr-A^{-1}\tau_{0}\ve^{2}-2\ve^{2}R^{2}.
\end{split}
\]
Now, combining the upper bound of $\ti{\vp}_{j,r}(p)$ and lower bound of $\ti{\vp}_{i,r}(p)$, we see that
\[
\begin{split}
\ti{\vp}_{i,r}(p)-\ti{\vp}_{j,r}(p)
\geq \left(\sup_{B_{i,r}(p)}u_{\infty}-\sup_{B_{j,\frac{r}{2}}(p)}u_{\infty}\right)
-6Kr-(\tau_{0}+1)A^{-1}\ve^{2}+\ve^{2}R^{2}.
\end{split}
\]
Using $B_{j,\frac{r}{2}}(p)\subset B_{i,r}(p)$ and choosing $A=\frac{3\tau_{0}+3}{R^{2}}$,
\[
\ti{\vp}_{i,r}(p)-\ti{\vp}_{j,r}(p)
\geq -6Kr+\frac{2\ve^{2}R^{2}}{3}.
\]
Choosing $r_{0}=\frac{\ve^{2}R^{2}}{36K}$, we are done.
\end{proof}

\subsection{Proof of Theorem \ref{subsolution}}
Now we are in a position to prove Theorem \ref{subsolution}.

\begin{proof}[Proof of Theorem \ref{subsolution}]
Let $\ti{\vp}$ be the regularized maximum of $(B_{i,3R},\ti{\vp}_{i,r})$. Lemma \ref{gluing lemma} shows that $\ti{\vp}\in C^{\infty}(X)$ by \cite[Lemma I.5.18]{Demaillybook}. Thanks to \eqref{gluing eqn 1},
\[
\alpha_{\ti{\vp}} \in \Gamma_{\omega,\theta_{0},\Theta_{0}} \ \
\text{on $X$},
\]
which implies $\ti{\vp}$ is the required subsolution.
\end{proof}

\section{Proof of $(4)\Rightarrow(1)$ in Theorem \ref{main theorem}}\label{Section: main result 2}
In this section, we give the proof of $(4)\Rightarrow(1)$ in Theorem \ref{main theorem}.

\begin{lma}\label{coerciveness lemma}
There exists $C(\alpha,\omega,X)$ such that for any $\vp\in\H$,
\[
\begin{split}
& \int_{0}^{1}\int_{X}(-\vp)
\left(a_{0}\Im(\alpha_{t\vp}+\sqrt{-1}\omega)^{n}-\omega^{n}\right)dt \\
\geq {} & \frac{a_{0}}{C}\int_{X}(-\vp)\left(\Im(\alpha_{\vp}+\sqrt{-1}\omega)^{n}
-\Im(\alpha+\sqrt{-1}\omega)^{n}\right)-C.
\end{split}
\]
\end{lma}

\begin{proof}
We assume without loss of generality that $\sup_{X}\vp=0$. Since $\vp\in\PSH(X,\chi)$ from Lemma \ref{quasi-plurisubharmonic}, then $\int_{X}(-\vp)\omega^{n}\leq C$. It suffices to show
\[
\int_{0}^{1}\int_{X}(-\vp)\Im(\alpha_{t\vp}+\sqrt{-1}\omega)^{n}dt
\geq \frac{1}{C}\int_{X}(-\vp)\Im(\alpha_{\vp}+\sqrt{-1}\omega)^{n}.
\]
Thanks to Lemma \ref{s Im control},
\[
\begin{split}
& \int_{0}^{1}\int_{X}(-\vp)\Im(\alpha_{t\vp}+\sqrt{-1}\omega)^{n}dt \\
\geq {} & \int_{\frac{1}{2}}^{1}\int_{X}(-\vp)\Im(\alpha_{t\vp}+\sqrt{-1}\omega)^{n}dt \\
\geq {} & \frac{1}{C}\int_{X}(-\vp)\Im(\alpha_{\vp}+\sqrt{-1}\omega)^{n},
\end{split}
\]
as required.
\end{proof}

Now we are in a position to prove $(4)\Rightarrow(1)$ in Theorem \ref{main theorem}.

\begin{proof}[Proof of $(4)\Rightarrow(1)$ in Theorem \ref{main theorem}]
For any $\vp\in\H$, we consider the twisted dHYM flow starting from $\vp$. By Theorem~\ref{smooth convergence}, if the dHYM equation admits a solution, then for $\ve$ sufficiently small, the twisted dHYM flow admits a long-time solution and will converge to a stationary solution $\hat{\vp}_{\ve}$. Fixing such an $\ve$ and using $\J_{\ve}$ is decreasing along twisted dHYM flow, we obtain
\[
\J_{\ve}(\vp) \geq \J_{\ve}(\hat{\vp}_{\ve}) \geq -C.
\]
By definitions of $\J_{0}$ and $\J_{\ve}$,
\[
\J_{0}(\vp) = \J_{\ve}(\vp)+\ve\int_{0}^{1}\int_{X}(-\vp)\left(a_{0}\Im(\alpha_{t\vp}+\sqrt{-1}\omega)^{n}-\omega^{n}\right)dt.
\]
Thanks to Lemma \ref{coerciveness lemma},
\[
\J_{0}(\vp)
\geq \frac{a_{0}\ve}{C}\int_{X}(-\vp)\left(\Im(\alpha_{\vp}+\sqrt{-1}\omega)^{n}
-\Im(\alpha+\sqrt{-1}\omega)^{n}\right)-C.
\]
Combining this with $\J=\sin(\theta_{0})\J_{0}$ (see Lemma \ref{J J 0 relationship}), we obtain $\J$ is coercive.
\end{proof}

\end{document}